\documentclass[a4paper,12pt,intlimits,oneside]{amsart}
\usepackage{amsmath}
\usepackage{amsthm}
\usepackage{latexsym}
\usepackage{amssymb}
\usepackage{xcolor}
\usepackage{dsfont}
\usepackage{mathrsfs}
\usepackage[colorlinks=true]{hyperref}
\usepackage{graphicx,epstopdf}

\numberwithin{equation}{section}
\numberwithin{figure}{section}
\def\expo_#1{{\rm e}^{#1}}

\def\e{\varepsilon}
\def\R{{\mathbb R}}
\def\C{{\mathbb C}}
\def\D{{\mathbb D}}
\def\T{{\mathbb T}}

\def\s{\vskip 0.25cm\noindent}

\def\build#1_#2^#3{\mathrel{\mathop{\kern 0pt#1}\limits_{#2}^{#3}}}
\def\td_#1,#2{\mathrel{\mathop{\build\longrightarrow_{#1\rightarrow #2}^{}}}}
\def\e{\varepsilon}
\newcommand{\ben}{\begin{equation}}
\newcommand{\een}{\end{equation}}
\newcommand{\beno}{\begin{eqnarray*}}
\newcommand{\eeno}{\end{eqnarray*}}

\newtheorem{theorem}{Theorem}
\newtheorem{corollary}{Corollary}
\newtheorem{proposition}{Proposition}
\newtheorem{lemma}{Lemma}
\newtheorem{remark}{Remark}
\newtheorem{example}{Example}

\date{November 10, 2019}
\begin{document}
\title[On a damped Szeg\H{o} equation]{On a damped Szeg\H{o} equation \\(with an appendix in collaboration with Christian Klein)}
\author[P. G\'erard]{Patrick G\'erard}
\address{ Laboratoire de Math\'ematiques d'Orsay, Universit\'e Paris-Sud, CNRS, Universit\'e Paris-Saclay, 91405 Orsay, France}
\author[S. Grellier]{Sandrine Grellier}
\address{Institut Denis Poisson, D\'epartement de Math\'ematiques, Universit\'e d'Orleans, 45067
Orl\'eans Cedex 2, France}
\begin{abstract}
We investigate how  damping  the lowest Fourier mode modifies the dynamics of the cubic Szeg\H{o} equation. We show that there is a nonempty open subset of initial data generating  trajectories with high Sobolev norms tending to infinity. In addition, we give a complete picture of this phenomenon on a reduced phase space of dimension $6$. An appendix is devoted to numerical simulations supporting the generalisation of  this picture to more general initial data.
\end{abstract}

\thanks{The authors are grateful to T. Alazard for drawing their attention to the introduction of damping in integrable PDEs.}
\maketitle
\tableofcontents
\section{Introduction}
In the last decade, a number of papers have tried to display  growth of Sobolev norms of high regularity for solutions of globally wellposed nonlinear Hamiltonian partial differential equations. This question, raised by Bourgain in \cite{Bo1}, \cite{Bo2} for the defocusing nonlinear Schr\"odinger equation on the torus, led to several contributions constructing solutions with a small initial Sobolev norm of high regularity and a big Sobolev norm at some later time, see \cite{CKSTT}, \cite{GGens}, \cite{GGAnPDE}, \cite{Gu}, \cite{GuK}, \cite{HP}, \cite{GHP}, \cite{GHHMP}, \cite{GLPR}. The actual existence of unbounded trajectories was proved in \cite{Po}, \cite{H},  \cite{HPTV},  \cite{GGAst},  \cite{X}, \cite{X2}, \cite{CG}, \cite{T}, \cite{Gfields}. In this paper, we intend to study a case  where a weak damping can promote unbounded trajectories in Sobolev spaces with high regularity. This unexpected phenomenon will be displayed in the particular case of a weak damping applied to the cubic Szeg\H{o} equation. It would be interesting to investigate how weak damping can perturb other Hamiltonian dynamics in the same way.\s
The cubic Szeg\H{o} equation was introduced in \cite{GGens} as a toy model of  degenerate nondispersive Hamiltonian dynamics. A natural phase space is the intersection of the Sobolev space $H^{\frac 12}(\T )$ with the Hardy space $L^2_+(\T )$ made of square integrable functions on the circle with only nonnegative Fourier modes. This phase space will be denoted by $H^{\frac 12}_+(\T )$. The equation reads
\begin{equation}\label{szego}
i\partial_tu=\Pi (|u|^2u)\ ,
\end{equation}
where $\Pi $ is the orthogonal projector from $L^2(\T )$ onto $L^2_+(\T )$.  An important property of  equation \eqref{szego} is 
the existence of a Lax pair structure, leading to action--angle variables \cite{GGinv}. On the other hand, the conservation laws do not control high Sobolev regularity of the solutions,
allowing for long term infinitely many transitions  between low and high frequencies, as proved in \cite{GGAst} -- see also the lecture notes \cite{Gfields}. However, such solutions are very unstable. The goal of this paper is to investigate how these properties are changed when adding a damping term, which breaks both the Hamiltonian and the integrable structures. Such a problem for integrable systems seems very difficult. In order to make it more amenable, we choose a specific damping term which keeps part of the above structure. This leads to the following equation.
\begin{equation}\label{DS}
i\partial_tu+i\alpha(u\vert 1)=\Pi(|u|^2u)\ ,
\end{equation}
where $\alpha>0$ is a given parameter, which could be made equal to $1$ after a scaling transformation, and $(u\vert 1)$ is the Fourier coefficient $\hat u(0)$ of $u$.  Note that the momentum 
$$M(u):=(Du\vert u)=\sum_{k\ge 1}k|\hat u(k)|^2$$ is preserved by the flow. An easy modification of the arguments in \cite{GGens} shows that \eqref{DS} is globally wellposed 
on $H^s_+:=L^2_+\cap H^s$ for every $s\ge \frac 12$. Our goal is to study the behaviour of solutions of \eqref{DS} as $t\to +\infty $, in particular the growth of Sobolev norms $H^s$ for $s>\frac 12$. Recall from \cite{GGAst} that, for a dense $G_\delta $ subset of initial data in $L^2_+\cap C^\infty $, 
the solutions of \eqref{szego} satisfy, for every $s>\frac 12$,
$$\limsup_{t\to +\infty} \Vert u(t)\Vert _{H^s}=+\infty \ ,\ \liminf_{t\to +\infty} \Vert u(t)\Vert _{H^s}<+\infty\ .$$
Furthermore, this subset has an empty interior, since it does not contain any trigonometric polynomial. It turns out that the introduction of the damping term drastically modifies this asymptotic behaviour. Indeed, our main result is the following.
\begin{theorem}\label{main}
There exists an open subset $\Omega $ of $H^{\frac 12}_+$ such that, for every $s>\frac 12$, $\Omega \cap H^s_+$ is not empty and every solution $u$ of \eqref{DS} with $u(0)\in \Omega \cap H^s_+$ satisfies
$$\Vert u(t)\Vert _{H^s}\td_t,{+\infty}+\infty \ .$$
\end{theorem}

In fact, we obtain an explicit sufficient condition on initial data which drives to an exploding orbit in $H^s_+$ (see Theorem \ref{Exploding} below). 
\s
Theorem \ref{main} calls for  a number of natural questions.
\begin{enumerate}
\item Is the open set $\Omega \cap H^s_+$ dense in $H^s_+$ ? 
\item What is the rate of the growth of $\Vert u(t)\Vert _{H^s}$ ?
\end{enumerate}
At this stage, we do not have a complete answer to these questions. Nevertheless, the evolution of \eqref{DS} admits an invariant finite dimensional submanifold on which a complete description of the dynamics is available, providing a precise answer to questions (1), (2) in this particular setting. We denote by $\mathcal W$ the subset of functions $u$ on $\T $ of the form
$$u(x)=b+\frac{c\, {\rm e}^{ix}}{1-p{\rm e}^{ix}}\ ,$$
where $b, c, p\in \C $, $c\ne 0$, $|p|<1$. Note that $\mathcal W$ is a closed submanifold of dimension $6$ in $H^{\frac 12}_+$.  One can prove that, if $u$ is a solution of \eqref{DS} with $u(0)\in \mathcal W$, then $u(t)\in \mathcal W$ for every $t\in \R$
(see section 2 below).  Given $M>0$, we define the following hypersurface of $\mathcal W$
$$\mathcal E_M=\{u\in\mathcal W;\;M(u)=M\}.$$
We also denote by $\mathcal C_M$ the circle made of functions of the form
$$u(z)=cz\ ,\ |c|^2=M\ ,$$
which is a closed orbit for \eqref{DS}. 
\begin{theorem}\label{W}
For every $M>0$ and every $\alpha >0$, there exists a codimension $2$ submanifold  $\Sigma_{M,\alpha }$  of $\mathcal E_M$, disjoint from $\mathcal C_M$, invariant by the evolution of \eqref{DS}, such that
\begin{itemize}
\item If $u$ solves \eqref{DS} with $u(0)\in \mathcal E_M\setminus (\mathcal C_M\cup \Sigma_{M,\alpha })$, then, for every $s>\frac 12$, as $t\to +\infty $,
$$\| u(t)\| _{H^s}\sim c(s,\alpha ,M)\, t^{s-\frac 12}$$
with $c(s,\alpha ,M)>0$.
\item  If $u(0)\in \Sigma_{M,\alpha }$, then $u(t)$ tends to $\mathcal C_M$ at $t\to +\infty$, and 
$${\rm dist}(u(t), \mathcal C_M)\simeq {\rm e}^{-\lambda (\alpha ,M)t}\ ,$$
with $\lambda (\alpha ,M)>0\ .$
\end{itemize}
\end{theorem}
Let us say a few words about the ingredients of the proofs of Theorems \ref{main} and \ref{W}. The first important feature of the damped Szeg\H{o} equation \eqref{DS} is that the $L^2$ norm is a Lyapunov functional,
$$\frac{d}{dt}\Vert u(t)\Vert _{L^2}^2+2\alpha |(u(t)\vert 1)|^2=0\ .$$
Using  LaSalle's invariance principle associated to this identity, one infers that limit points of $u(t)$ as $t\to +\infty $ in the weak $H^{\frac 12}$ topology are initial data of solutions $v$ of  \eqref{szego} satisfying $(v(t)\vert 1)=0$ for every $t\in \R $.\\
The second important argument relies on the Lax pair structure for the Szeg\H{o} equation \eqref{szego}, which is given by
$$\frac{dH_u}{dt}=[B_u,H_u]\ ,\ \frac{dK_u}{dt}=[C_u,K_u]\ ,$$
where $H_u, K_u$ are Hankel operators associated to $u$, and $B_u, C_u$ are antiselfadjoint operators. It turns out that, though the identity for $H_u$ does not hold 
anymore for solutions of the damped Szeg\H{o} equation \eqref{DS}, identity for $K_u$ remains valid for \eqref{DS}. As a consequence, the spectrum of the positive trace class operator $K_u^2$ is conserved by the dynamics of \eqref{DS}. \\
The connection between the two above arguments is made thanks to a characterization of initial data of solutions $v$ of  \eqref{szego} satisfying $(v(t)\vert 1)=0$ for every $t\in \R $, in terms of the spectral theory 
of $H_v$ and $K_v$. If $u$ is a solution whose $H^s$ norm does not tend to $+\infty$, this allows to calculate  the limit of the $L^2$ norm of  $u(t)$  in terms of the spectrum of $K_u^2$, leading to Theorem \ref{main}.\\
As for Theorem \ref{W}, the Lax pair identity for $K_u$ implies that the dynamics of \eqref{DS} preserves functions $u$ such that operator $K_u$ has a given finite rank. Manifold $\mathcal W$ precisely corresponds to operators $K_u$ of rank $1$. A careful study of the ODE system defined by \eqref{DS} on $\mathcal W$ then leads to Theorem \ref{W}.
\s
This paper is organised as follows. Section 2 is devoted to recalling important facts about the Szeg\H{o} equation and Hankel operators, and to establishing general properties of solutions of the damped Szeg\H{o} equation. Theorem \ref{main} is proved in Section 3,
and Theorem \ref{W} is proved in Section 4.
\s
Let us mention that some introductory material to the Szeg\H{o} equation can be found in \cite{Gfields}, and in \cite{GGchine}, where the results of this paper were announced.
\section{Generalities on the damped and undamped Szeg\H{o} equations} 
In the following, we denote by $t\mapsto S(t)u_0$ (respectively  by  $t\mapsto S_\alpha(t)u_0$) the solution of the Szeg\H{o} equation \eqref{szego} (respectively of the damped Szeg\H{o} equation \eqref{DS}) with initial datum $u_0$.

\subsection{The Lyapunov functional}
As emphasized in the introduction, an important tool in the study of the damped Szeg\H{o} equation (\ref{DS}) is the existence of a Lyapunov functional. Precisely, the following lemma holds.
\begin{lemma}\label{limit}
Let $u_0\in H^{1/2}_+(\T)$. Then, for any $t\in \R$,
\begin{equation}\label{Lyapunov}\frac d{dt}\Vert S_\alpha(t) u_0\Vert^2_{L^2}+2\alpha |(S_\alpha u_0(t)\vert 1)|^2=0.
\end{equation} As a consequence,
 $t\mapsto\Vert S_\alpha(t)u_0\Vert_{L^2}$ is decreasing, and $|(S_\alpha(t)u_0\vert 1)|$ is square integrable on $[0,+\infty)$, tending to zero as $t$ goes to $+\infty $.
\end{lemma}
\begin{proof}
Denote by $u(t):=S_\alpha(t)u_0$ the solution of \eqref{DS} with $u(0)=u_0$. Observe first that $t\mapsto\Vert u(t)\Vert_{L^2}$ decreases:
\begin{eqnarray*}
\frac d{dt}\Vert u(t)\Vert^2_{L^2}&=& 2{\rm Re}(\partial_t u\vert u)=2{\rm Im} (i\partial_t u\vert u)\\
&=&2{\rm Im} (\Pi(|u|^2u)\vert u)-2\alpha {\rm Im} (i((u\vert 1)\vert u))\\
&=&-2\alpha |(u(t)\vert 1)|^2
\end{eqnarray*}
Hence, $t\mapsto\Vert u(t)\Vert^2_{L^2}$ admits a limit at infinity and since
$$\Vert u(t)\Vert^2_{L^2}=\Vert u_0\Vert^2_{L^2}-2\alpha\int_0^t|(u(s)\vert 1)|^2 ds$$  we deduce the finiteness of 
$$\int_0^\infty|(u(s)\vert 1)|^2 ds.$$
On the other hand, we claim that $$\frac d{dt}  |(u(t)\vert 1)|^2$$ is bounded. Indeed
\begin{eqnarray*}
\frac d{dt}  |(u(t)\vert 1)|^2&=& 2{\rm Re} ((\partial_t u\vert 1)(1\vert u))\\
&=& 2{\rm Im}(-i\alpha (u\vert 1)(1\vert u))+2{\rm Im}((\Pi(|u|^2u\vert 1)(1\vert u)))\\
&=&-2\alpha |(u\vert 1)|^2+2{\rm Im}((u^2\vert u)(1\vert u)).
\end{eqnarray*}
but $$|(u(t)\vert 1)|\le \Vert u\Vert_{L^2}$$
and 
$$|(u^2(t)\vert u(t))|\le \Vert u\Vert_{L^2}\times \Vert u\Vert_{L^4}^2\le \Vert u\Vert_{L^2}\times\Vert u\Vert_{H^{1/2}}^2\le  \Vert u_0\Vert_{L^2}(M+\Vert u_0\Vert_{L^2}^2).$$
From both observations, we conclude that $|(u(t)\vert 1)|$ tends to zero as $t$ goes to infinity.
\end{proof}

From Lemma \ref{limit} and the conservation of the momentum, the $H^{1/2}$ norm of $S_\alpha(t)u_0$ remains bounded as $t\to +\infty $, hence one can consider limit points $u_\infty$ of $S_\alpha (t)u_0$ for the weak topology of $H^{1/2}$ as $t\to +\infty $.
Another general lemma describes more precisely these limit points, according to  LaSalle's invariance principle.
\begin{proposition}\label{WeakLimit}
Let $u_0\in H^{1/2}(\T)$. Any $H^{1/2}$- weak limit point $u_\infty$ of $(S_\alpha(t) u_0)$ 
as $t\to +\infty$  satisfies $(S_\alpha(t)u_\infty\vert 1)=0$ for all $t$. In particular, $S_\alpha(t)u_\infty$ solves the cubic Szeg\H{o} equation -- in other words $S_\alpha(t)u_\infty=S(t)u_\infty$.
\end{proposition}
\begin{proof}
Denote by $Q$ the limit of the decreasing non-negative function $t\mapsto\Vert S_\alpha(t)u_0\Vert_{L^2}^2$. By the weak continuity of the flow in $H^{1/2}_+(\T)$,
$$u(t+t_n)=S_\alpha(t)u(t_n)\to S_\alpha(t)u_\infty$$ weakly in $H^{1/2}$ as $n\to \infty$. Hence, thanks to the Rellich theorem, $$\Vert u(t+t_n)\Vert_{L^2}^2\to \Vert S_\alpha(t)u_\infty\Vert ^2_{L^2}$$ as $n$ tends to infinity. On the other hand, by Lemma \ref{limit},  $$\Vert u(t+t_n)\Vert_{L^2}^2\to Q$$ so eventually, for every $t\in \R$,  $$\Vert S_\alpha(t)u_\infty\Vert ^2_{L^2}=\Vert u_\infty\Vert ^2_{L^2},$$ or $$\frac d{dt}\Vert S_\alpha(t)u_\infty\Vert^2_{L^2}=0.$$ 
Recall that, from \eqref{Lyapunov},
$$\frac{d}{dt}\Vert S_\alpha(t)u_\infty\Vert_{L^2}^2=-2\alpha|(S_\alpha(t)u_\infty\vert 1)|^2.$$
It forces $(S_\alpha(t)u_\infty\vert 1)=0$ for all $t$. Hence $S_\alpha(t)u_\infty=S(t)u_\infty$ is a solution to the Szeg\H{o} equation without damping. 
\end{proof}
In order to characterize $u_\infty$, we need to recall some results about the cubic Szeg\H{o} equation.
\subsection{Hankel operators and the Lax pair structure}
In this paragraph, we recall some basic facts about Hankel operators and the special structure of the cubic Szeg\H{o} equation \eqref{szego}. We keep the notation of \cite{GGAst} and we refer to it for  details. For $u\in H^{\frac 12}_+$, we denote by $H_u$ the Hankel operator of symbol $u$ namely
$$H_u:\left\{\begin{array}{cll}
L^2_+(\T)&\to& L^2_+(\T)\\
f&\mapsto& \Pi(u\overline f)
\end{array}\right.$$
It is well known that, for $u$ in $H^{\frac 12}_+$, $H_u$ is Hilbert-Schmidt with $${\rm Tr}(H_u^2)=\sum_{k\ge 0} (k+1)|\hat u(k)|^2=\Vert u\Vert_{L^2}^2+M(u).$$ One can also consider the shifted Hankel operator $K_u$ corresponding to $H_{S^*u}$ where $S^*$ denotes the adjoint of the shift operator $Sf(x):={\rm e}^{ix}f(x)$.
This shifted Hankel operator is Hilbert-Schmidt as well, with  $${\rm Tr}(K_u^2)=\sum_{k\ge 0} k|\hat u(k)|^2=M(u)$$ Observe in particular that $\Vert u\Vert_{L^2}^2={\rm Tr}(H_u^2)-{\rm Tr}(K_u^2)$.

A crucial property of the cubic Szeg\H{o} equation is its Lax pair structure. Namely, if $u$ is a smooth enough solution to \eqref{szego}, then there exists two antiselfadjoint operators $B_u, \, C_u$ such that 
\begin{equation}\label{LaxPair}\frac{d}{dt}H_u=[B_u,H_u],\; \;\frac{d}{dt}K_u=[C_u,K_u].\end{equation}
Classically, these equalities imply that $H_{u(t)}$ and $K_{u(t)}$ are isometrically equivalent to $H_{u(0)}$ and $K_{u(0)}$ (see \cite{GGAst} for instance). In particular, both spectra of $H_u$ and $K_u$ are preserved by the cubic Szeg\H{o} flow. It motivated the study of the spectral properties of both Hankel operators that we recall here.\\
For $u\in H^{\frac 12}_+$, let $(s_j^2)_{j\ge 1}$ be the strictly decreasing sequence of positive eigenvalues of $H_u^2$ and $K_u^2$. Following the terminology of \cite{GGAst}, $\sigma^2$ is an $H$-dominant eigenvalue (respectively $K$-dominant eigenvalue) if $\sigma^2$ is an eigenvalue of $H_u^2$ (respectively of $K_u^2$) with $u\not\perp \ker(H_u^2-\sigma^2I)$ (respectively  with $u\not\perp \ker(K_u^2-\sigma^2I)$). From the min-max formula and the fact that $K_u^2=H_u^2-(\cdot\vert u)u$, it is possible to prove that the $s_{2j-1}^2$ correspond to $H$-dominant eigenvalues of $H_u^2$ while the $s_{2j}^2$ correspond to $K$-dominant eigenvalues of $K_u^2$.  Furthermore,  the eigenvalues of $H_u^2$ and of $K_u^2$ interlace and, as a consequence, if 
$m_j:=\dim \ker (H_u^2-s_j^2I)$ and  $\tilde m_j:=\dim \ker (K_u^2-s_j^2I)$,  then $$m_{2j-1}=\tilde m_{2j-1}+1\; \text{ as }\tilde m_{2j}=m_{2j}+1.$$ \\
To complete the spectral analysis of these Hankel operators, we need to recall the notion of Blaschke product. A function $b$ is a Blaschke product of degree $m$ if $$b(x)={\rm e^{i\varphi}}\prod_{j=1}^m \frac{{\rm e^{ix}}-p_j}{1-\overline {p_j}{\rm e^{ix}}}$$ for some $p_j\in \C$ with $|p_j|<1$, $j=1$ to $m$.
As proved  in \cite{GGAst} --- see also \cite{GP} for a generalisation to non compact Hankel operators---, for any $H$-dominant eigenvalue $s_{2j-1}^2$, there exists a Blaschke product $\Psi_{2j-1}$ of degree $m_{2j-1}-1$ such that, if $u_j$ denotes the orthogonal projection of $u$ on the eigenspace $\ker(H_u^2-s_{2j-1}^2I)$, then $$\Psi_{2j-1}H_u(u_j)=s_{2j-1}u_j.$$
Analogously, for any $K$-dominant eigenvalue $s_{2k}^2$, there exists a Blaschke product $\Psi_{2k}$ of degree $\tilde m_{2k}$ such that, if $\tilde u_k$ denotes the orthogonal projection of $u$ on $\ker(K_u^2-s_{2k}^2I)$ then $$K_u(\tilde u_k)=\Psi_{2k}s_{2k}\tilde u_k.$$

We proved in \cite{GGAst} that the sequence $((s_j^2),(\Psi_j))$ characterizes $u$, and that it provides a system of action-angle variables for the Hamiltonian evolution \eqref{szego}.  Namely, 
if $u(0)$ has spectral coordinates $((s_j^2),(\Psi_j))$ then $u(t)$ has spectral coordinates $$((s_j^2),({\rm e}^{i(-1)^js_j^2t}\Psi_j)).$$ 
We are now in position to characterize the asymptotics of the damped Szeg\H{o} equation. We first remark that the equation inherits one Lax pair, the one related to the shifted Hankel operator $K_u$. It comes easily from the fact the shifted Hankel operator associated to a constant symbol is identically $0$ and so, if $u(t):=S_\alpha(t) u_0$, then
\begin{equation}\label{LaxK}\frac d{dt}K_u=K_{-i\Pi(|u|^2u)}+\alpha K_{(u\vert 1)}=[C_u,K_u]
\end{equation}
where $C_u$ is the antiselfadjoint operator given by \eqref{LaxPair}.
As a usual consequence, $K_{S_\alpha(t) u_0}$ is unitarily equivalent to $K_{u_0}$ and for instance, the class of symbol $u$ with $K_u$ of fixed finite rank is preserved by the damped Szeg\H{o} flow. In particular \begin{equation}\label{SetW}
\mathcal W:=\left \{u(x)=b+\frac {c\, {\rm e}^{ix}}{1-p{\rm e}^{ix}},\; b,c,p\in\C,  c\neq 0,\; |p|<1\right \}
\end{equation}
 is invariant by the flow since it corresponds to the set of symbol whose shifted Hankel operators are of rank $1$.\\
 Another consequence is the following result.
\begin{theorem}\label{u(0)=0}
The solutions $u=u(t,\cdot)$ of the cubic Szeg\H{o} equation satifying $(u(t)\vert 1)=0$ for all $t$ are characterized by the property $(\Psi_{2j-1}\vert 1)=0$, for all the Blaschke products $\Psi_{2j-1}$'s corresponding to the $H$-dominant eigenvalues $s_{2j-1}^2$ of $H_u^2$. In particular, the $H$-dominant eigenvalues are at least of multiplicity $2$ and hence, are eigenvalues of $K_u^2$. Furthermore, if $\{\sigma_k^2\}_k$ denotes the strictly decreasing sequence of the eigenvalues of $K_u^2$, one has
$$\Vert u(t)\Vert_{L^2}^2=\sum_k(-1)^{k-1}\sigma_k^2.$$
\end{theorem}
\begin{proof}
Let us write $u=\sum u_j$ where $u_j$ is the orthogonal projection of $u$ onto the eigenspace $E_u(s_{2j-1}):=\ker(H_u^2-s_{2j-1}^2I)$ associated to the $H$-dominant eigenvalue $s_{2j-1}^2$. By the spectral analysis of the Hankel operator recalled above, there exists a Blaschke product $\Psi_{2j-1}$ of degree $m-1$ where $m$ is the dimension of $E_u(s_{2j-1})$ with  $s_{2j-1} u_j=\Psi_{2j-1} H_u(u_j)$. The evolution of the Blaschke product is given by 
$$\Psi_{2j-1}(t)={\rm e}^{-is_{2j-1}^2 t}\Psi_{2j-1}(0).$$
Computing $(u(t)\vert 1)$, we get, for all $t\in \R$,
$$0=(u(t)\vert 1)=\sum(u_j(t)\vert 1)=\sum_j\frac{{\rm e}^{-is_{2j-1}^2 t}}{s_{2j-1}}(\Psi_{2j-1}(0)\vert 1)\Vert u_j\Vert_{L^2}^2.$$ It implies that $(\Psi_{2j-1}(0)\vert 1)=0$ so that the degree of $\Psi_{2j-1}$ is at least $1$ and hence, the multiplicity of $s_{2j-1}^2$ is at least $2$.
From the interlacement property, this eigenvalue is also an eigenvalue for $K_u^2$. Let $\{\sigma_k^2\}_k$ denote the strictly decreasing sequence of the eigenvalues of $K_u^2$. We denote by $m_k$ the multiplicity of  $\sigma_k^2$ as an eigenvalue of $H_u^2$ and by $\tilde m_k$ its multiplicity as a eigenvalue of $K_u^2$. From the interlacement property, if $k$ is odd, $m_k=\tilde m_k+1$ and if $k$ is even, $m_k=\tilde m_k-1$. We now compute the $L^2$ norm of $u(t)$:
$$\Vert u(t)\Vert_{L^2}^2={\rm Tr }H_u^2-{\rm Tr}K_u^2=\sum m_k\sigma_k^2-\sum \tilde m_k \sigma_k^2=\sum_k(-1)^{k-1}\sigma_k^2.$$

\end{proof}

\section{Exploding trajectories}

In this section, we consider trajectories of \eqref{DS} in $H^s$, $s>\frac 12$, along which the $H^s$ norm of $u(t)$ tends to infinity as $t\to +\infty$. 
Let us define the functional $$F(u)=\sum_k(-1)^{k-1}\sigma_k^2$$ where $(\sigma_k^2)_k$ is the strictly decreasing sequence of positive eigenvalues of $K_{u}^2$. We prove the following result.
\begin{theorem}\label{Exploding}
Let $s>\frac 12$. If $u_0\in H^s_+$ satisfies
\begin{itemize}
\item either $ \Vert u_0\Vert_{L^2}^2 <F(u_0) $,
\item or $ \Vert u_0\Vert_{L^2}^2 =F(u_0)$ and $(u_0\vert 1)\ne 0$,
\end{itemize}
then the  $H^s$--norm of the solution of the damped Szeg\H{o} equation $$\Vert u(t)\Vert_{H^s}=\Vert S_\alpha(t) u_0\Vert_{H^s}$$ tends to $+\infty$ as $t$ tends to  $+\infty$.
\end{theorem}

\begin{proof}
Let us proceed by contradiction and assume that there exists a sequence $t_n\to +\infty $ such that $u(t_n):=S_\alpha(t_n)u_0$ is bounded in $H^s$. 
We may assume that $u(t_n)$ is weakly convergent to some $u_\infty $ in $H^s_+$. By the Rellich theorem, the convergence is strong in $H^{\frac 12}_+$,
and 
\begin{equation}\label{MM}
M(u_\infty )=M(u_0)=\sum_{\sigma^2 \in \Sigma (u_0)}m(\sigma )\sigma ^2\ ,
\end{equation}
where $\Sigma(u_0)$ denotes the set of eigenvalues of $K_{u_0}^2$ and $m(\sigma )$ the multiplicity of $\sigma ^2\in \Sigma(u_0)$. By the Lax pair structure, the eigenvalues of $K_{u(t_n)}^2$ are the same as the eigenvalues of $K_{u_0}^2$, with the same multiplicities, hence every eigenvalue $\sigma^2 $ of $K_{u_\infty}^2$ must belong to $\Sigma (u_0)$, with a multiplicity not bigger than $m(\sigma )$. In view of  identity  \eqref{MM}, we infer that
$$\Sigma (u_\infty )=\Sigma (u_0)\ ,$$
with the same multiplicities. On the other hand, from Proposition \ref{WeakLimit}, we know that $u_\infty$ generates a solution of the cubic Szeg\H{o} equation which is orthogonal to $1$ at every time. Consequently, Theorem \ref{u(0)=0} gives
$$\| u_\infty \|_{L^2}^2=F(u_0)\ .$$
Since the $L^2$-norm of the solution is decreasing by Lemma \ref{limit}, $ \| u_0\|_{L^2}^2\ge \| u_\infty \|_{L^2}^2\ .$  Hence, $ \| u_0\|_{L^2}^2\ge F(u_0)$.
 If $ \| u_0\|_{L^2}^2=F(u_0)$ then $\Vert S_\alpha(t)u_0\Vert _{L^2}$ remains constant and necessarily, by the Lyapunov functional identity \eqref{Lyapunov}, $(S_\alpha (t)u_0\vert 1)=0$ so that in particular, $(u_0\vert 1)=0$. Hence, the case $ \| u_0\|_{L^2}^2=F(u_0)$ and $(u_0\vert 1)\neq 0$ drives to an exploding orbit in $H^s$ as well as the case $ \| u_0\|_{L^2}< F(u_0)$. It ends the proof of Theorem \ref{Exploding}. 
\end{proof}

\subsection{The case of Blaschke products}
As a particular case of initial datum satisfying $ \Vert u_0\Vert_{L^2}^2 =F(u_0)$ and $(u_0\vert 1)\ne 0$, we consider initial datum given by a Blaschke product.

\begin{corollary}
Let $\Psi$ be a Blaschke product with $(\Psi\vert 1)\neq 0$ then $\Vert S_\alpha(t)\Psi\Vert_{H^s}$ tends to $+\infty$ with $t$ for any $s>\frac 12$.
\end{corollary}
\begin{proof}
Observe that, as $\Psi$ is inner, $H_\Psi^2(\Psi)=H_\Psi(1)=\Psi$ so that $1$ is an eigenvalue of $H_\Psi^2$ with eigenvector $\Psi$. From the spectral analysis done in \cite{GGAst}, and in particular from Lemma 3.5.2, one obtains the explicit description of the eigenspace corresponding to $1$. If $\Psi$ is of degree $N$, this eigenspace is of dimension $N+1$ hence $1$ is $H$-dominant, of multiplicity $N+1$ and the representation of $\Psi$ through the non-linear Fourier transform is $\Psi$ itself. In particular, the rank of $H_\Psi$ is $N+1$. On the other hand,  from the interlacing property, $1$ is a singular value of $K_\Psi$ of multiplicity $N$ and the rank of $K_\Psi$ is $N$.  Hence  $K_{\Psi}$ has only $1$ as possible non zero singular value and $F(\Psi)=1$. As $\Vert \Psi\Vert_{L^2}=1$ we get the norm explosion as a corollary of Theorem \ref{Exploding}.
\end{proof}
\subsection{An open condition}
As a second corollary of Theorem  \ref{Exploding}, we get the following result, which implies Theorem \ref{main}.
\begin{corollary} Denote by $\Omega $ the interior in $H^{\frac 12}_+$ of the set of $u_0\in H^{\frac 12}_+$ such that $ \Vert u_0\Vert_{L^2}^2 <F(u_0) $. 
For every $s>\frac 12$,  $\Omega \cap H^s_+$ of $H^{s}_+(\T)$ is not empty, and 
every solution $u$ of  with $u(0)\in\Omega \cap H^s_+$ satisfies
$$\Vert u(t)\Vert_{H^s}\to \infty$$
as $t$ tends to $+\infty $. 
\end{corollary}
\begin{proof}
By elementary perturbation theory, it is easy to prove that 
function $F$ is continuous at those $u$ of $H^{1/2}_+(\T)$ such that $K_u^2$ has simple non zero spectrum. Furthermore, in the
particular case
$$u(x)=\frac{{\rm e}^{ix}}{1-p\, {\rm e}^{ix}}, \; p\in \D,\; p\neq 0,$$ it is easy to check that $K_u^2$ has rank one with $\frac 1{(1-|p|^2)^2}$ as simple eigenvalue.
As $\Vert u\Vert_{L^2}^2=\frac 1{1-|p|^2}$, this function belongs to $\Omega $, and moreover it belongs to every $H^s$. 
In view of Theorem \ref{Exploding}, this completes the proof.
\end{proof}

We end this section by giving a simple example of functions in $\Omega$:

\begin{example}
The set of functions $u_0$ whose nonzero eigenvalues of $H_{u_0}^2$ and $K_{u_0}^2$ are all simple, and form the decreasing square summable list
$$\rho_1>\sigma_1>\rho_2>\sigma_2>\dots $$
with
$$\sum_j \rho_j^2<2\sum_{k\, {\rm odd}}\sigma_k^2$$
is a subset of $\Omega$.
\end{example}

\section{A special case}
In this section, we restrict ourselves to the set 
$$\mathcal W:=\left \{u(x)=b+\frac {c\, {\rm e}^{ix}}{1-p{\rm e}^{ix}},\; b,c,p\in\C,  c\neq 0,\; |p|<1\right \}$$ introduced in \eqref{SetW}. As recalled in section 2.2, $\mathcal W$ corresponds to rational functions $u$ with shifted Hankel operator $K_u$ of rank $1$ and hence it is preserved by the damped Szeg\H{o} flow.  It is straightforward to check that the system on variables $b,c,p$ reads
\begin{eqnarray}\label{system}
\left \{ 
\begin{array}{r c l}
i(\dot b+\alpha b)&=&(|b|^2+2M(1-|p|^2))b +Mc\overline p\\
i\dot c&=&(2|b|^2+M)c+2M(1-|p|^2)bp\\
i\dot p&=&M(1-|p|^2)p+c\overline b
\end{array}
\right .
\end{eqnarray}

 In this section, we provide a panorama of the dynamics of the damped Szeg\H{o} equation on $\mathcal W$.\\
In particular, we prove that the set of functions $u_0$ such that, for some $t$, condition
\begin{equation}\label{CondExplode}\Vert S_\alpha(t)u_0\Vert_{L^2}^2<F(u_0).\end{equation}  is satisfied, is a dense open subset of $\mathcal W$ on which the growth of the $H^s$ norm of $ S_\alpha(t)u_0$   as $t$ tends to $+\infty$ is of order $t^{s-\frac12}$. Moreover we indicate the structure of the complement of this set. Let us observe that when $u_0\in\mathcal W$, $F(u_0)=M(u_0)$ so that condition \eqref{CondExplode} reads $\Vert S_\alpha(t)u_0\Vert_{L^2}^2<M(u_0)$. \\

Let us recall the statement given in the introduction in a more precise form.\\
 
 Given $M>0$, we define the following five dimensional hypersurface of $\mathcal W$
$$\mathcal E_M=\{u\in\mathcal W;\;M(u)=M\}.$$
We also denote by $\mathcal C_M$ the circle made of functions of the form
$$u(x)=c\, {\rm e}^{ix}\ ,\ |c|^2=M\ $$ which are invariant by the damped Szeg\H{o} flow \eqref{DS} as a periodic orbit.

In the next statement, we normalise the $H^s$ norm as follows.
\begin{equation}\label{normeHs}
\| u\|_{H^s}^2=\sum_{k=0}^\infty (1+k^2)^s|\hat u(k)|^2\ .
\end{equation}

%We first remark that a sufficient condition for a solution of the damped Szeg\H{o} equation to explode in $H^s$--norm is that its $L^2$ norm decreases to $0$.
%Indeed, Rellich theorem implies the unboundedness of $\Vert S_\alpha(t)u_0\Vert_{H^s}$. Observe that real interpolation gives the rate of unboundedness 
%$$\Vert u\Vert_{H^{1/2}}\le C\Vert u\Vert_{L^2}^{1-1/2s}\Vert u\Vert ^{1/2s}_{H^s}$$ hence
%$$\Vert u\Vert_{H^s}\ge CM^s\Vert u\Vert^{1-2s}_{L^2}.$$
\begin{theorem}\label{Wbis}
For every $M>0$ and every $\alpha >0$, there exists a codimension $2$ submanifold  $\Sigma_{M,\alpha }$  of $\mathcal E_M$, disjoint from $\mathcal C_M$, invariant by the action of $S_\alpha(t)$, such that $\Sigma_{M,\alpha }\cup \mathcal C_M$ is closed and
\begin{itemize}
\item If $u_0\in \mathcal E_M\setminus (\mathcal C_M\cup \Sigma_{M,\alpha })$, then, for every $s>\frac 12$, as $t\to +\infty $,
$$\| S_\alpha(t)u_0\| _{H^s}^2\sim c^2(s,\alpha ,M)\, t^{2s- 1}$$
with $$c^2(s,\alpha ,M):=\Gamma(2s+1)M^{4s-1}\left ( \frac{\alpha^2+M^2}{2\alpha}   \right )^{1-2s}>0\ .$$
%Furthermore, if we define complex numbers $b(t),c(t),p(t)$ by writing $$S_\alpha(t)u_0=b(t)+\frac{c(t)z}{1-p(t)z},$$
%there exist $\theta_0,\theta_1\in \T,\ell \in \R$ with
%\begin{eqnarray*}
%p(t)&=&{\rm exp}\left ( -i\kappa \log t+i\theta_0+\frac{i\xi\log t}{t} +\left (i\xi +i\ell-\frac{\kappa ^2}{\alpha +iM}\right )\frac 1t+\mathcal O\left (\frac{(\log t)^2}{t^2}\right )  \right )\\
%c(t)&=&\frac{\kappa}{\sqrt M t}{\rm exp}\left (-iMt-i\frac{M}{\alpha }\log t+i\theta_1+\mathcal O\left (\frac {\log t}{t}\right )\right ) \\
%b(t)&=&\sqrt{\frac \kappa \alpha }\frac 1t {\rm exp}\left (-iMt+i\left (\kappa -\frac{M}{\alpha }\right )\log t+i(\theta_1-\theta_0)+\mathcal O\left (\frac {\log t}{t}\right )\right )\ .
%\end{eqnarray*}
\item If $u_0\in \Sigma_{M,\alpha }$, then $S_\alpha(t)u_0$ tends to $\mathcal C_M$ as $t\to +\infty$, and 
$${\rm dist}(S_\alpha(t)u_0, \mathcal C_M)\simeq {\rm e}^{-\lambda (\alpha ,M)t}\ ,$$
with $\lambda (\alpha ,M)>0\ .$
%More precisely, there exist $(\beta_\infty, \theta, \varphi )\in (0,\infty)\times \T\times \T $ such that, as $t\to +\infty$,
%\begin{eqnarray*}
%b(t)&\sim &  \sqrt{\beta_\infty}\, {\rm e}^{-\frac{a+\alpha}{2}t-itM\left (1+\frac{\alpha }{a}\right )+i\varphi}\ ,  \\
%c(t)&\sim & \sqrt M {\rm e}^{-itM+i\theta}\ ,\\
%p(t)&\sim & \sqrt{\frac{\beta_\infty}{M}}\left (\frac{\sqrt{a^2-\alpha ^2}-i(\alpha -a)}{2M}-1\right )\, {\rm e}^{-\frac{a+\alpha}{2}t+itM\frac{\alpha }{a}+i(\theta -\varphi)}\ .
%\end{eqnarray*}
%Conversely, for every $(\beta_\infty, \theta, \varphi )\in (0,\infty)\times \T\times \T $, there exists a unique trajectory 
%$$u(t,z)=b(t)+\frac{c(t)z}{1-p(t)z}$$
%satisfying the above asymptotic properties.

\end{itemize}
\end{theorem}

Before giving the proof of this result, let us make some basic observations. We write
$$u(t):=S_\alpha(t)u_0=b(t)+\frac{c(t)\, {\rm e}^{ix}}{1-p(t)\, {\rm e}^{ix}},$$ $b,c,p\in\C$.
Since the momentum is a conservation law, $$M(u)=\frac{|c|^2}{(1-|p|^2)^2}$$ remains a constant, denoted by $M$. Denoting by $Q$ the limit as $t$ goes to infinity of $\Vert u(t)\Vert_{L^2}^2=|b|^2+\frac{|c|^2}{(1-|p|^2)}=|b|^2+M(1-|p|^2)$, we get that $M(1-|p|^2)$ tends to $Q$ (from Lemma \ref{limit}, $|b|\to 0$). In particular, $|p(t)|^2$ admits a limit in $[0,1]$ as $t$ tends to infinity.\\
We claim that the following alternative holds:
\begin{itemize}
\item either $\lim_{t\to+\infty}|p(t)|^2=1$, $Q=0$ and $$\Vert u(t)\Vert_{H^s}^2\sim \Gamma(2s+1)\frac{|c|^2}{(1-|p|^2)^{2s+1}}=\frac {\Gamma(2s+1)\, M}{(1-|p|^2)^{2s-1}}\to \infty$$ for any $s>1/2$.
\item or $\lim_{t\to +\infty}|p(t)|^2<1$ and we claim that 
$$\lim_{t\to\infty}|p(t)|^2=0.$$
\end{itemize}
Indeed, let us consider the latter case. As $\lim_{t\to +\infty}|p(t)|^2<1$, the set $\{ u(t)\ ,\ t\ge 0\}$ is relatively compact in any $H^s(\T)$. Denote by $u_\infty$ any $H^{1/2}$ limit of $(S_\alpha(t)u_0)$. As $\mathcal W$ is closed in $H^{1/2}_+$, $u_\infty$ belongs to $\mathcal W$. From Proposition \ref{WeakLimit}, $S_\alpha(t)u_\infty$ solves the cubic Szeg\H{o} equation and so equals $S(t)u_\infty$.
 As $S(t)u_\infty \in \mathcal W$ and $(S(t)u_\infty\vert 1)=0$ for all $t$, $S(t)u_\infty$ is necessarily of the form $x\mapsto \frac{c_\infty(t) {\rm e}^{ix}}{1-p_\infty(t) {\rm e}^{ix}}$. From the equation satisfied by $b$ in system \eqref{system}, 
$$i(\dot b+\alpha b)=(|b|^2+2M(1-|p|^2))b +Mc\overline p,$$ since $b_\infty(t)=0$, we obtain $p_\infty(t)=0$. 

Hence, $\lim_{t\to+\infty}|p(t)|^2=0$, and $$\lim_{t\to+\infty}\Vert u(t)\Vert_{L^2}^2=Q=M(1-\lim_{t\to +\infty} |p(t)|^2)=M\ ,$$ in particular 
$\Vert u_0\Vert^2_{L^2}\ge M$.\\
Eventually, we get the following alternative.
\begin{proposition}\label{altern}
If $u$ is a solution of \eqref{DS} on $\mathcal W$, either $u(t)$ is bounded in any $H^s$, $s>1/2$, as $t\to +\infty $, and  $\Vert u_0\Vert^2_{L^2}\ge M$, or the trajectory is exploding  in the sense that $\Vert u(t)\Vert_{H^s}$ tends to infinity for any $s>\frac 12$.
\end{proposition}
So we can rephrase the statement of Theorem \ref{Wbis} in view of these observations. Theorem \ref{Wbis} claims that those data of momentum $M$ for which $p(t)\to 0$ as $t\to +\infty$ form the disjoint union of the circle $\mathcal C_M$ and of a three dimensional submanifold $\Sigma_{M,\alpha}$,  which is a union of  trajectories converging exponentially to $\mathcal C_M$ as $t\to +\infty $. 
Furthermore, outside of this set, $1-|p(t)|^2 \to 0$ as $t\to +\infty$, with a universal rate 
$$(1-|p(t)|^2)\sim \frac{\rho (\alpha, M)}{t},$$
where $\rho (\alpha, M)=\frac{\alpha^2+M^2}{2\alpha M^2}>0$.
\s
We split the proof of Theorem \ref{Wbis} into two parts. The first one is a careful analysis of the case $|p(t)|\to 1$ through the differential equations corresponding to \eqref{DS} on some reduced variables. The second part consists in reducing  the system to a scattering problem.

\subsection{The growth of Sobolev norms}

In this section, we consider a solution of \eqref{DS}
$$u(t,x)=b(t)+\frac{c(t)\, {\rm e}^{ix}}{1-p(t)\, {\rm e}^{ix}}$$
such that $|p(t)|^2\to 1$ as $t\to +\infty $, of momentum
$$M=\frac{|c(t)|^2}{(1-|p(t)|^2)^2}\ .$$
In order to avoid the gauge and translation invariances, we appeal to the following reduced variables, 
$$\beta :=|b|^2\ ,\ \gamma :=M(1-|p|^2)\ ,\ \zeta :=Mc\overline b\overline p\ ,$$
which, from system \eqref{system}, satisfy the following reduced system,
\begin{eqnarray}\label{redsys}
\, \left \{
\begin{array}{r c l}
\dot \beta +2\alpha \beta &=&2{\rm Im}\zeta \\
\dot \gamma &=&-2{\rm Im}\zeta \\
\dot \zeta +(\alpha +iM)\zeta &=&(3i\gamma -i\beta )\zeta -2i\beta \gamma M\\
       && +i\gamma ^2(M-\gamma +3\beta )
\end{array}
\right.
\end{eqnarray}
Notice that 
$$|\zeta |^2=(M-\gamma)\gamma ^2\beta \ ,\ \beta \geq 0\ ,\  M\geq \gamma >0\ .$$
From Lemma \ref{limit}, we already know that $\beta (t)\to 0$ as $t\to +\infty $ and that
$$\int_0^\infty \beta (t)\, dt <+\infty \ .$$
Our task is to prove, under the additional assumption $\gamma (t)\to 0$ as $t\to +\infty$, 
that 
$$\gamma (t)\sim \frac{\kappa }{t}$$
with $\kappa =\kappa (\alpha ,M)>0$. 

As a first step, let us establish that
$$\int _0^\infty \gamma (t)^2\, dt <+\infty \ .$$ 
We write the equation on $\zeta $ in system \eqref{redsys} as
\begin{equation}\label{eqzeta}
\frac{\dot \zeta }{\alpha +iM}+\zeta =f+r \ ,
\end{equation}
where
\begin{eqnarray*}
&& f:=i\frac{\gamma ^2}{\alpha +iM}\left (M-\gamma +3\frac{\zeta }{\gamma}\right )\ ,\\
&&r:=-i\frac{\beta }{\alpha +iM}\left (\zeta +2\gamma M -3\gamma^2\right )\ .
\end{eqnarray*}
Notice that $r\in L^1(\R_+)$ and that, as $t\to +\infty $,
$${\rm Im}f(t)\sim \frac{\alpha M}{\alpha ^2+M^2}\gamma (t)^2>0\ .$$
Integrating from $0$ to $T$ the imaginary part of both sides of \eqref{eqzeta}, we obtain, using $\dot \gamma =-2{\rm Im}\zeta $,
$$\int_0^T{\rm Im}f(t)\, dt =O(1)$$
as $T\to +\infty $. This provides
$$\int _0^\infty \gamma (t)^2\, dt <+\infty \ .$$ 
The second step consists in coming back to \eqref{eqzeta} and integrating the imaginary part of both sides from $t$ to $+\infty$. Using $|\zeta |=O(\gamma \sqrt{\beta})$ and $r=O(\beta \gamma )$  , we infer
$$
\gamma (t)(1+o(1))=\frac{2\alpha M}{\alpha ^2+M^2}\left (\int_t^\infty \gamma (s)^2\, ds\right ) (1+o(1))+O\left (\int_t^\infty \beta (s)\gamma (s)\, ds\right )\ .
$$
Using again that $\beta \in L^1(\R_+)$, this yields 
\begin{equation}\label{eqgamma}
\gamma (t)=\frac{1}{\kappa }\left (\int_t^\infty \gamma (s)^2\, ds\right ) (1+o(1))+o\left (\sup_{s\ge t}\gamma (s)\right )\ ,
\end{equation}
with 
\begin{equation}\label{kappa}
\kappa :=\frac{\alpha ^2+M^2}{2\alpha M}\ .
\end{equation}
Using the monotonicity of 
\begin{equation}\label{Gamma}
\Gamma (t):=\int_t^\infty \gamma (s)^2\, ds\ ,
\end{equation}
equation \eqref{eqgamma} leads to
$$\sup_{s\ge t}\gamma (s)=\frac{1}{\kappa }\left (\int_t^\infty \gamma (s)^2\, ds\right ) (1+o(1))\ ,$$
and, coming back to \eqref{eqgamma}, we finally obtain
\begin{equation}\label{eqgammabis}
\gamma (t)= \frac{1}{\kappa }\left (\int_t^\infty \gamma (s)^2\, ds\right ) (1+o(1))\ .
\end{equation}
This equation can be written as an ODE in function $\Gamma $ introduced above in \eqref{Gamma},
$$\dot \Gamma +\frac{1}{\kappa ^2}\Gamma ^2 (1+o(1))=0\ ,$$
which can be solved as
$$\Gamma (t)=\frac{\kappa ^2}{t}(1+o(1))$$
and, coming back to \eqref{eqgammabis}, 
$$\gamma (t)=\frac{\kappa }{t}(1+o(1)).$$
Eventually, one gets 
\begin{eqnarray*}
\Vert u(t)\Vert_{H^s}^2&\sim &\frac{\Gamma (2s+1)\, M}{(1-|p(t)|^2)^{2s-1}}=\Gamma(2s+1)\, M^{2s}\gamma(t)^{1-2s}\\
&\sim &\Gamma(2s+1)\, M^{2s}\kappa^{1-2s}t^{2s-1}
\end{eqnarray*} which, in view of the expression \eqref{kappa} of $\kappa$,  is the expected result.

\subsection{The stable manifold of the periodic orbit}

We come to the second part of the theorem, characterising the trajectories of \eqref{DS} in $\mathcal W$ which converge to $\mathcal C_M$ as $t\to +\infty$. Since $\mathcal C_M$ is a trajectory itself, we may focus on trajectories $\{ u(t)\}$ of momentum $M$ which are disjoint of $\mathcal{C}_M$, and which satisfy
$$\forall t\geq 0\ ,\ \| u(t)\|_{L^2}^2 \geq M .$$
At this stage we are not sure that such  trajectories exist. However we are going to establish a necessary condition on the asymptotic behaviour of $(b(t),c(t),p(t))$ as $t\to +\infty $. 
\s
As a first step, we will use a linearisation procedure that we now illustrate on a``baby example". 
\begin{proposition}\label{baby}
Let  $u_0^\e(x)={\rm e}^{ix}+\varepsilon$ with $\e \in \R$.
For $|\e |$ small enough and $\e \ne 0$, the trajectory $S_\alpha (t)u_0^\e $ is exploding.
\end{proposition}
\begin{proof}
First observe that $\Vert u_0^\e \Vert_{L^2}^2=1+\varepsilon^2$ and $F(u_0^\e )=M(u_0^\e )=1$. We are going to  prove that, for some $T>0$, for $\varepsilon$ small enough, $\Vert S_\alpha (T)u_0^\e\Vert_{L^2}^2\le 1-\varepsilon^2$, so that Proposition \ref{baby} will follow from Theorem \ref{Exploding}.\\
 Let us prove the claim.
We linearize the flow $S_\alpha (t)$ around the solution ${\rm e}^{i(x-t)}$ for $\e =0$ by writing
$$S_\alpha(t) u_0^\e ={\rm e}^{-it}({\rm e}^{ix}+\varepsilon v+\varepsilon ^2w^\e)$$ where $v$ satisfies
$$i(\partial_tv+\alpha(v\vert1))=v+\Pi({\rm e}^{2ix}\overline v), \;\; v(0,{\rm e}^{ix})=1$$ and, for every $T>0$, for every $s$, there exists $C_{s,T}$ such that 
$$\forall t\in[0,T],\; \Vert w^\e (t)\Vert_{H^s}\le C_{s,T}.$$
From the Lyapunov functional \eqref{Lyapunov}, for any $t\in [0,T]$,
\begin{eqnarray*}
\Vert S_\alpha(t)u_0^\e \Vert_{L^2}^2&=&\Vert u_0^\e \Vert_{L^2}^2-2\alpha\int_0^t|(S_\alpha(s)u_0^\e \vert 1)|^2 ds\\
&=&1+\varepsilon^2-2\alpha \varepsilon^2\int_0^t(|(v(s)\vert 1)|^2+\mathcal O_T(\varepsilon))ds.
\end{eqnarray*}
Let us write 
$$v(t,x)=q_0(t)+q_2(t){\rm e}^{i2x}$$
with
\begin{eqnarray*}
i(\dot q_0+\alpha q_0)&=&q_0+\overline q_2\ ,\\
i\dot q_2&=&q_2+\overline q_0\ ,\\
q_0(0)=1\ &,&\ q_2(0)=0\ .
\end{eqnarray*}
Let us focus on $q_0=(v\vert 1)$. Deriving the first equation, we are left with the following second order ODE,
$$q_0''+\alpha \dot q_0-i\alpha q_0=0\ ,$$
with the initial data
$$q_0(0)=1\ ,\ \dot q_0(0)=-(\alpha +i)\ .$$
The solutions of the characteristic equation
$$\lambda ^2+\alpha \lambda -i\alpha =0$$
are given by
$$\lambda_{\pm}=\frac{-\alpha \pm (a+i\sqrt{a^2-\alpha ^2})}{2}\ .$$
where 
$$
a:=\left ( \frac{\alpha\sqrt{\alpha ^2+16}+\alpha ^2}{2} \right )^{\frac 12}>\alpha \ .
$$

Notice that ${\rm Re}(\lambda_+)>0$ while ${\rm Re}(\lambda_-)<0$. 
This leads to
$$q_0(t)=A_+{\rm e}^{\lambda _+t}+A_-{\rm e}^{\lambda_-t}\ ,$$
with 
\begin{eqnarray*}
A_+&=&\frac{\dot q_0(0)-\lambda_-q_0(0)}{\lambda_+-\lambda_-}=-\frac{(\alpha +i+\lambda_-)}{(a+i\sqrt{a^2-\alpha ^2})}\ ,\\
A_-&=&\frac{\lambda_+q_0(0)-\dot q_0(0)}{\lambda_+-\lambda_-}=\frac{(\alpha +i+\lambda_+)}{(a+i\sqrt{a^2-\alpha ^2})}\ 
\end{eqnarray*}
In particular, 
$$\int_0^t|q_0(s)|^2 ds$$ tends to $+\infty$ with $t$. Let us come back to the expression of 
the $L^2$ norm of our solution,
$$\Vert S_\alpha(t)u_0^\e \Vert_{L^2}^2=1+\varepsilon^2-2\alpha \varepsilon^2\int_0^t(|q_0(s)|^2+\mathcal O_T(\varepsilon))ds.$$
Fix $T>0 $ such that 
$$2\alpha \int_0^T|q_0(s)|^2\, ds\ge 3\ .$$
Fix $\varepsilon >0$ small enough so that $2\alpha T \mathcal O_T(\varepsilon)\leq 1$, we obtain
 $$\Vert S_\alpha (T)u_0^\e \Vert_{L^2}^2\le 1-\varepsilon^2\ ,$$
 which is the claim.
\end{proof}
Let us complete the proof of the second part of Theorem \ref{Wbis}.
First we extend to this general context the  notation introduced in the proof of Proposition \ref{baby}, 
\begin{equation}\label{a}
a:=\left ( \frac{\sqrt{\alpha ^4+16M^2\alpha ^2}+\alpha ^2}{2} \right )^{\frac 12}>\alpha \ .
\end{equation}
Notice that
\begin{equation}\label{magie}
a\sqrt{a^2-\alpha ^2}=2M\alpha \ .
\end{equation}
\begin{lemma}\label{necessary}
If $\{ u(t) \} $ is a trajectory of \eqref{DS} in $\mathcal W$, of momentum $M$, disjoint of $\mathcal C_M$ and such that $$\forall t\geq 0\ ,\ \| u(t)\|_{L^2}^2 \geq M ,$$
then $b(t)\ne 0$ for every $t$, and there exists $\theta \in \T$ such that, as $t\to +\infty $,
\begin{eqnarray*}
{\rm e}^{-itM}\sqrt M\, \frac{\overline {p(t)}}{b(t)}&\to &{\rm e}^{-i\theta }\,
\left (\frac{\sqrt{a^2-\alpha ^2}+i(\alpha -a)}{2M}-1\right )\ ,\\ 
{\rm e}^{itM}c(t)&\to &\sqrt M\, {\rm e}^{i\theta }\ .
\end{eqnarray*}
\end{lemma}
\begin{proof} Since 
$$\| u(t)\|_{L^2}^2=|b(t)|^2+M(1-|p(t)|^2), $$
the condition on $u$ reads
\begin{equation}\label{bp}
\forall t, |b(t)|^2\geq M|p(t)|^2\ .
\end{equation}
Since the trajectory is disjoint from $\mathcal C_M$, $b$ and $p$ cannot cancel at the same time, and therefore $b(t)\ne 0$ for every $t$. From the ODE on $c$ in \eqref{system}, we have
\begin{equation}\label{eq:c}
i\frac{d}{dt}\left ({\rm e}^{itM}c(t)\right )={\rm e}^{itM}\, (2|b(t)|^2c(t)+2M(1-|p(t)|^2)b(t)p(t) )\ .\end{equation}
Recall that $|b|^2\in L^1(\R_+)$. Therefore \eqref{bp} implies that $|p|^2\in L^1(\R_+)$, hence 
$$\frac{d}{dt}\left ({\rm e}^{itM}c(t)\right )\in L^1(\R_+)\ ,$$
and consequently 
$${\rm e}^{itM}c(t)\to c_\infty \ .$$
Notice that, since $p(t)\to 0$, $|c_\infty |^2=M$, hence there exists $\theta \in \T$ such that
$$c_\infty =\sqrt M\,  {\rm e}^{i\theta }\ .$$
In order to establish the other condition, we fix $T\gg 1$ and we set
$$\e =\e (T):=|b(T)|\ .$$
Integrating from $T$ to $+\infty $  the identity
$$\frac{d}{dt}(|b(t)|^2+M(1-|p(t)|^2))=-2\alpha |b(t)|^2\ ,$$
we obtain
\begin{equation}
|b(T)|^2-M|p(T)|^2=2\alpha \int_T^\infty |b(t)|^2\, dt\ .
\end{equation}
Consequently,
$$|b(T)|=\e \ ,\ |p(T)|\leq \frac{|b(T)|}{\sqrt M}=O(\e )\ ,\ |c(T)-c_\infty {\rm e}^{-iTM}|=O\left (\int_T^\infty |b(t)|^2\, dt\right )=O(\e ^2)\ ,$$
where the last estimate comes from integrating \eqref{eq:c}.
In other words, in any $H^s$ space,
\begin{eqnarray*}
&&{\rm dist}\left (u(T), c_\infty {\rm e}^{-iTM}{\rm e}^{ix}\right )=O(\e )\ ,\\
&& {\rm dist}\left ( u(T), b(T)+ c_\infty {\rm e}^{-iTM}{\rm e}^{ix} + c_\infty {\rm e}^{-iTM}p(T) {\rm e}^{2ix} \right )=O(\e ^2)\ .
\end{eqnarray*}
At this stage, we are going to describe $u(t+T)=S_\alpha(t)u(T)$ for $t\ge 0$ by means of the linearised equation, exactly as in the proof of Proposition \ref{baby}. We obtain
$$S_\alpha(t)u(T)(x)= {\rm e}^{-itM}\left (c_\infty {\rm e}^{-iTM}{\rm e}^{ix}+\e v(t,x)+\e ^2w(t,x)\right )\ ,$$
where $v$ is the solution of the linearised problem
\begin{eqnarray*}
i(\partial_tv+\alpha (v|1))+Mv&=&2Mv+\Pi \left (c_\infty ^2{\rm e}^{-2iTM}{\rm e}^{2ix}\overline v\right )\ ,\\
v\left (0,x\right )&=&\frac{b(T)}{\e}+c_\infty {\rm e}^{-iTM}\frac{p(T)}{\e}{\rm e}^{2ix}\ ,
\end{eqnarray*}
and, for every $R>0$, there exists $C_{s,R}$ independent of $T$ such that
$$\forall t\in [0,R], \| w(t)\|_{H^s}\leq C_{s,R}\ .$$
This leads to
\begin{eqnarray*}
\| S_\alpha(t)u(T)\|_{L^2}^2-M&=&\| u(T)\|_{L^2}^2-M-2\alpha \int_0^t |(S(s)u(T)|1)|^2\, ds\\
&=&|b(T)|^2-M|p(T)|^2-2\alpha \e^2\int_0^t(|(v(s)|1)|^2+O_R(\e ))\, ds
\end{eqnarray*}
for every $t\in [0,R]$. In other words, the condition $\| u(t+T)\|^2\ge M$ reads as follows : for every $R>0$, there exists $c_R>0$ such that, for every $t\in [0,R]$, for every $T>0$,
\begin{equation}\label{ineq}
|b(T)|^2-M|p(T)|^2-2\alpha \e^2\int_0^t|(v(s)|1)|^2\, ds\geq -c_R\e ^3\ .
\end{equation}
Let us compute $v(t)$. From the linearized problem, we find
$$v\left (t,x\right )=q_0(t)+q_2(t){\rm e}^{2ix}\ ,$$
with
\begin{eqnarray*}
i(\dot q_0+\alpha q_0)&=&Mq_0+c_\infty ^2{\rm e}^{-2iTM}\overline q_2\ ,\\
i\dot q_2&=&Mq_2+c_\infty ^2{\rm e}^{-2iTM}\overline q_0\ ,\\
q_0(0)=\frac{b(T)}{\e}\ &,&\ q_2(0)=c_\infty {\rm e}^{-iTM}\frac{p(T)}{\e}\ .
\end{eqnarray*}
Let us focus on $q_0$. Deriving the first equation, we have
\begin{eqnarray*}
i(q_0''+\alpha \dot q_0)&=&M\dot q_0+c_\infty ^2{\rm e}^{-2iTM}\overline {\dot q_2}\ ,\\
&=&M\dot q_0+c_\infty ^2{\rm e}^{-2iTM}\left (iM\overline q_2+i\overline c_\infty ^2{\rm e}^{2iTM} q_0\right )\\
&=&-M(\alpha +iM)q_0+iM^2q_0\\
&=&-\alpha Mq_0\ .
\end{eqnarray*}
We are left with the following second order ODE,
$$q_0''+\alpha \dot q_0-i\alpha Mq_0=0\ ,$$
with the initial data
$$q_0(0)=\frac{b(T)}{\e }\ ,\ \dot q_0(0)=-(\alpha +iM)\frac{b(T)}{\e }-iMc_\infty {\rm e}^{-iTM}\frac{\overline {p(T)}}{\e}\ .$$
Again, the solutions of the characteristic equation
$$\lambda ^2+\alpha \lambda -iM\alpha =0$$
are given by
$$\lambda_{\pm}=\frac{-\alpha \pm (a+i\sqrt{a^2-\alpha ^2})}{2}\ .$$
Again, ${\rm Re}(\lambda_+)>0>{\rm Re}(\lambda_-)$ and this leads to
$$q_0(t)=A_+{\rm e}^{\lambda _+t}+A_-{\rm e}^{\lambda_-t}\ ,$$
with 
\begin{eqnarray*}
A_+&=&\frac{\dot q_0(0)-\lambda_-q_0(0)}{\lambda_+-\lambda_-}=-\frac{(\alpha +iM+\lambda_-)b(T)+i Mc_\infty {\rm e}^{-iTM}\overline {p(T)}}{\e (T)(a+i\sqrt{a^2-\alpha ^2})}\ ,\\
A_-&=&\frac{\lambda_+q_0(0)-\dot q_0(0)}{\lambda_+-\lambda_-}=\frac{(\alpha +iM+\lambda_+)b(T)+i Mc_\infty {\rm e}^{-iTM}\overline {p(T)}}{\e (T)(a+i\sqrt{a^2-\alpha ^2})}\ 
\end{eqnarray*}
Rephrasing \eqref{ineq}, we must have, for every $R>0$, for every $t\in [0,R]$, for every $T>0$,
\begin{equation}\label{ineq2}
1-M\frac{|p(T)|^2}{\e (T)^2}-2\alpha \int_0^t|q_0(t)|^2\, dt \geq -c_R\, \e (T)\ .
\end{equation}
Computing the integral of $|q_0|^2$, we infer, for some uniform constant $B>0$, 
$$\frac{2\alpha }{a-\alpha }|A_+(T)|^2{\rm e}^{(a-\alpha )R}\leq c_R\, \e (T)+ B\ .$$
Taking the upper limit of both sides as $T\to \infty$, we infer
$$\frac{2\alpha }{a-\alpha } \limsup_{T\to \infty} |A_+(T)|^2{\rm e}^{(a-\alpha )R}\leq B\ , $$
and finally, making $R\to +\infty $,
$$\limsup_{T\to \infty} |A_+(T)|^2=0\ .$$
Coming back to the expression of $A_+(T)$ and of $c_\infty $ above, this yields
$${\rm e}^{-iTM}\sqrt M \frac{\overline {p(T)}}{b(T)}\to i\frac{\alpha +iM+\lambda_-}{M}{\rm e}^{-i\theta}={\rm e}^{-i\theta }\,
\left (\frac{\sqrt{a^2-\alpha ^2}+i(\alpha -a)}{2M}-1\right )\ .$$
This completes the proof of Lemma \ref{necessary}.
\end{proof}
\begin{remark}\label{uniform}
For further reference, it is useful to observe that the result of Lemma \ref{necessary} can be made uniform. Indeed, assume that we have a sequence of solutions $\{ u_n\}$ of fixed momentum $M$ satisfying $\| u_n(t)\|_{L^2}^2\geq M$ for every $t, n$ and such that $(u_n(T)\vert 1)\td_T,{+\infty} 0$ uniformly with respect to $n$. Then we claim that the  two convergences established by Lemma \ref{necessary} are also uniform with respect to $n$. Indeed, this is straightforward for $c_n(t)$, since, as we already noticed,
$$|c_n(t)-c_{n,\infty}{\rm e}^{-itM}|\leq C\int_t^\infty |b_n(s)|^2\, ds \leq C'|b_n(t)|^2\ .$$
As for the second quantity, we just need to reproduce the above linearisation proof by observing that the linearisation is made near a compact sequence of solutions $\{ c_{n,\infty}{\rm e}^{-itM}\} $ with 
a family $\{ \e_n(T)\}$ of uniformly small parameters, so that, for every $R>0$,  the bound
$$\sup_{t\in [0,R]}\| w_n(t)\|_{H^s}\leq C_{s,R}$$
is uniform with respect to $n$. The result then follows from the estimate
$$\frac{2\alpha }{a-\alpha } \limsup_{T\to \infty} \sup_n|A_{n,+}(T)|^2{\rm e}^{(a-\alpha )R}\leq B\ . $$
\end{remark}
As a next step, we show that the corresponding solutions of the reduced  system on $\beta :=|b|^2, \delta :=M|p|^2=M-\gamma, \zeta :=Mc\overline{pb}$, satisfy a nonlinear scattering problem.
\begin{lemma}\label{bdz}
Let $\{ u(t) \} $ be a trajectory of \eqref{DS} in $\mathcal W$, of momentum $M$, disjoint of $\mathcal C_M$, such that $$\forall t\geq 0\ ,\ \| u(t)\|_{L^2}^2 \geq M .$$ Write
\begin{eqnarray*}
&&u(t,x)=b(t)+\frac{c(t){\rm e}^{ix}}{1-p(t){\rm e}^{ix}}\ ,\\
&&\beta (t):=|b(t)|^2\ ,\ \delta (t):=M|p(t)|^2\ ,\ \zeta (t):=Mc(t)\overline {b(t)p(t)}\ .
\end{eqnarray*}
Then there exists $\beta_\infty>0$ such that, as $t\to +\infty$,
\begin{eqnarray*}
\beta (t)&=&\beta_\infty \, {\rm e}^{-(a+\alpha)t}\left (1+O\left ({\rm e}^{-(a+\alpha )t }\right )\right )\ ,\\
\delta (t)&=&\frac{a-\alpha}{a+\alpha}\beta_\infty \, {\rm e}^{-(a+\alpha)t}\left (1+O\left ({\rm e}^{-(a+\alpha )t }\right )\right )\ ,\\
\zeta (t)&=&\left (\frac{\sqrt{a^2-\alpha^2}+i(\alpha -a)}{2}-M\right )\beta_\infty \, {\rm e}^{-(a+\alpha)t}\left (1+O\left ({\rm e}^{-(a+\alpha )t }\right )\right )\ .
\end{eqnarray*}
Conversely, for every $\beta_\infty >0$, there exists a unique solution $(\beta,\delta ,\zeta)$ of the reduced system
\begin{eqnarray*}
\left\{
\begin{array}{r c l}
\dot \beta +2\alpha \beta &=&2{\rm Im}\zeta \\
\dot \delta &=&2{\rm Im}\zeta \\
\dot \zeta +(\alpha -2iM)\zeta &=&-i(3\delta +\beta )\zeta +i(M-\delta)^2 (\delta +\beta )-2i\beta \delta (M-\delta)
\end{array}
\right.
\end{eqnarray*}
with the above asymptotic expansion as $t\to +\infty $. 

Furthermore, in this context, for every $C>0$, there exists $C'>0$ such that
\begin{itemize}
\item  If  $\beta (0)\geq C^{-1}$, then $\beta _\infty \geq (C')^{-1}$\ .
\item If $\beta (0)\leq C\ ,$ then $\beta_\infty \leq C'$.
\end{itemize}
\end{lemma} 
\begin{proof}
Setting $\delta :=M-\gamma$ in the reduced system \eqref{redsys} in $\beta, \gamma, \zeta$, we indeed obtain
\begin{eqnarray*}
\left\{
\begin{array}{r c l}
\dot \beta +2\alpha \beta &=&2{\rm Im}\zeta \ ,\\
\dot \delta &=&2{\rm Im}\zeta \ ,\\
\dot \zeta +(\alpha -2iM)\zeta &=&-i(3\delta +\beta )\zeta +i(M-\delta)^2 (\delta +\beta )-2i\beta \delta (M-\delta)\ 
\end{array}
\right.
\end{eqnarray*}
Furthermore, we already know that $(\beta (t),\delta (t),\zeta (t))\to (0,0,0)$ as $t\to +\infty$, and that
$$\int_0^\infty \beta (t)\, dt <\infty \ .$$
From the first two equations, we infer
$$\beta (t)-\delta (t)=2\alpha \int_t^\infty \beta (s)\, ds\ .$$
On the other hand, from Lemma \ref{necessary}, we have 
$$\frac{\delta (t)}{\beta (t)}\td_t,{+\infty} \left |\frac{\sqrt{a^2-\alpha ^2}+i(\alpha -a)}{2M}-1\right |^2 =
\frac{a-\alpha}{a+\alpha}\  ,$$
as an elementary calculation using \eqref{magie} shows.
Combining the above two informations, we obtain
$$\frac{\beta (t)}{\int_t^\infty \beta (s)\, ds}\td_t,{+\infty}a+\alpha \ ,$$
and consequently
$$\log \left (\int_t^\infty \beta (s)\, ds\right )=-(a+\alpha )t(1+o(1))\ .$$
In particular, for every $\e >0$, there exists $C_\e $ such that
$$\beta (t)\leq C_\e \, {\rm e}^{-(a+\alpha -\e)t}\ ,\ t\ge 0\ .$$
The same estimate holds for $\delta (t)$, and, in view of $|\zeta |=(M-\delta )\sqrt{\beta \delta }$, for $|\zeta (t)|$. Writing
$$X(t):=\left (\begin{array}{c}\beta (t)\\ \delta (t)\\ \zeta_R(t):={\rm Re}\zeta (t)\\ \zeta_I(t):={\rm Im}\zeta (t)\end{array}\right )\in \R^4\ ,$$ we observe that 
\begin{equation}\label{X}
\dot X+AX=Q(X)\ ,
\end{equation}
where 
\begin{eqnarray*}
A&:=&\left ( \begin{array}{cccc} 2\alpha & 0&0&-2 \\ 0&0&0&-2\\ 0&0&\alpha &2M\\ -M^2&-M^2&-2M &\alpha \end{array}\right )\ ,\\
 Q(X)&:=&\left (\begin{array}{c} 0\\ 0\\ (\beta +3\delta )\zeta_I\\
-(\beta +3\delta )\zeta_R-2M\delta ^2-4M\beta \delta +\delta ^3+3\beta \delta ^2  \end{array}\right )\ .
\end{eqnarray*}
An elementary calculation --- involving for instance $A-\alpha I$--- shows  that the four eigenvalues of $A$ are
$$\alpha \pm a\ ,\ \alpha \pm i\sqrt{a^2-\alpha ^2}\ .$$
Now recall that we have proved
$$|X(t)|\leq C_\e \, {\rm e}^{-(a+\alpha -\e)t}\ ,\ t\ge 0\ ,$$
hence, in view of the spectrum of $A$, since $Q(X)$ is a quadratic--cubic expression of $X$, we infer
$$|{\rm e}^{tA}[Q(X(t))] |\leq D_\e \, {\rm e}^{-(a+\alpha -2\e)t}\ ,\ t\ge 0\ .$$
From \eqref{X}, we have
$$\frac{d}{dt}[{\rm e}^{tA}X(t)]={\rm e}^{tA}Q[X(t)]\ .$$
Consequently there exists $X_\infty \in \R^4$ such that ${\rm e}^{tA}X(t)\to X_\infty$ as $t\to +\infty $. Integrating from $t$ to $+\infty$, we infer
\begin{equation}\label{Xint}
X(t)={\rm e}^{-tA}X_\infty -\int_t^\infty {\rm e}^{(s-t)A}[Q(X(s))]\, ds\ .
\end{equation}
In view of the spectrum of $A$, we have, for $s\ge t$,
$$\vert {\rm e}^{(s-t)A}[Q(X(s))]\vert \leq E_\e {\rm e}^{(a+\alpha)(s-t)-2(a+\alpha -\e)s}\ ,$$
so that
$$\left \vert \int_t^\infty {\rm e}^{(s-t)A}[Q(X(s))]\, ds \right \vert \leq F_\e {\rm e}^{-2(a+\alpha -\e)t}\ .$$
This implies in particular
$${\rm e}^{-tA}X_\infty=O\left ( {\rm e}^{-(a+\alpha)t}  \right )$$
for every $\e >0$, which, in view of the spectrum of $A$, imposes that $X_\infty$ is an eigenvector of $A$ for the eigenvalue $\lambda =\alpha +a$, hence there exists $\beta_\infty \in \R$ such that
$$X_\infty =\beta_\infty \left (\begin{array}{c} 1\\  \\ {\displaystyle \frac{a-\alpha}{a+\alpha}} \\ \\ {\displaystyle M\frac{\alpha -a}{a}}\\  \\ {\displaystyle \frac{\alpha -a}2}  \end{array}\right )\ .$$
Since $\beta (t)>0$, this imposes in particular $\beta _\infty>0$, so that, improving the remainder estimates by coming back to equation \eqref{Xint}, we conclude, using again \eqref{magie},
\begin{eqnarray*}
\beta (t)&=&\beta_\infty \, {\rm e}^{-(a+\alpha)t}\left (1+O\left ({\rm e}^{-(a+\alpha )t }\right )\right )\ ,\\
\delta (t)&=&\frac{a-\alpha}{a+\alpha}\beta_\infty \, {\rm e}^{-(a+\alpha)t}\left (1+O\left ({\rm e}^{-(a+\alpha )t }\right )\right )\ ,\\
\zeta (t)&=&\left (\frac{\sqrt{a^2-\alpha^2}+i(\alpha -a)}{2}-M\right )\beta_\infty \, {\rm e}^{-(a+\alpha)t}\left (1+O\left ({\rm e}^{-(a+\alpha )t }\right )\right )\ .
\end{eqnarray*}
Conversely, given any $\beta_\infty >0$, one can easily solve equation \eqref{Xint} by a fixed point argument on some interval $[T,\infty [$ for $T>0$ large enough, with the norm
$$\Vert X\Vert_T:=\sup_{t\ge T}{\rm e}^{(a+\alpha)t}|X(t)|\ .$$
 Then the extension to the whole real line is ensured by, say, the identities
$$|\zeta |^2=(M-\delta )^2\beta \delta \ ,\ \delta (t)+2\alpha \int_t^\infty \beta (s)\, ds=\beta (t)$$
which, combined with the first equation,  lead to
$$ | \dot \beta |=O(\beta )\ .$$
Finally, let us prove the last statement. From the fixed point argument mentioned above, it is easy to check that there exists $K>0$ such that, if $|X_\infty |$ is small enough, then 
$$|X(0)|\leq K|X_\infty |.$$
By contradiction, this proves that, if $\beta (0)\geq C^{-1}$, then $\beta _\infty \geq (C')^{-1}$.
The proof of the other inequality is slightly more delicate. Assume $\beta (0)\leq C$. 
As a first step, we are going to prove that $\beta (t)\to 0 $ as $t\to +\infty$ uniformly.
Indeed, since
\begin{equation}\label{bd}
\beta (t)=\delta (t)+2\alpha \int_t^\infty \beta (s)\, ds\ ,
\end{equation}
we infer
$$\frac{d}{dt}\left ({\rm e}^{2\alpha t}\int_t^\infty \beta (s)\, ds\right )=-\delta (t){\rm e}^{2\alpha t}\leq 0\ .$$
Hence
$$\int_t^\infty \beta (s)\, ds\leq \left (\int_0^\infty \beta (s)\, ds\right )\, {\rm e}^{-2\alpha t}\leq \frac{\beta(0)}{2\alpha}{\rm e}^{-2\alpha t}\ .$$
On the other hand, since $\beta (t)+M-\delta (t)$ is a decreasing function of $t$, we have
$$\beta (t)+M-\delta (t)\leq \beta (0)+M-\delta (0)\ ,$$
so that $\beta (t)\leq C+M\ .$ Coming back to the equation, we infer that there exists $B=B(\alpha, M, C)$ such that
$$|\dot \beta (t)|\leq B\ .$$
Consequently, for every $s\in \left [t, t+\frac{\beta (t)}{B}\right ]$, we have
$$\beta (s)\geq \beta (t)-B(s-t)$$
and therefore
$$\int_t^{t+\frac{\beta (t)}{B}}\beta (s)\, ds\geq \frac{\beta (t)^2}{B}-B\int_t^{t+\frac{\beta (t)}{B}}(s-t)\, ds=\frac{\beta (t)^2}{2B}\ .$$
In view of the estimate on $\int_t^\infty \beta $, we conclude
$$\beta (t)\leq K{\rm e}^{-\alpha t}$$
with $K=K(\alpha, M, C),$ which implies the uniform convergence of $\beta (t)$ to $0$. \\
Then, using Remark \ref{uniform} on uniformity in Lemma \ref{necessary}, we infer that, for every $\e >0$,  there exists $T=T(\alpha, M, C, \e)$ such that
$$\forall t\ge T\ ,\  \frac{a-\alpha-\e}{a+\alpha -\e} \leq \frac{\delta (t)}{\beta (t)}\leq \frac{a-\alpha+\e}{a+\alpha+\e }\ .$$
Coming back to identity \eqref{bd}, we infer, for every $t\ge T$,
$$\frac{d}{dt}\left ({\rm e}^{(a+\alpha -\e) t}\int_t^\infty \beta (s)\, ds\right )\leq 0\ ,$$
so that
$$\int_t^\infty \beta (s)\, ds\leq \left (\int_T^\infty \beta (s)\, ds\right )\, {\rm e}^{-(a +\alpha -\e )(t-T)}\leq \frac{\beta(0)}{2\alpha}{\rm e}^{-(a+\alpha -\e )(t-T)}\ .$$
Since
$$\beta (t)=\delta (t)+2\alpha \int_t^\infty \beta \leq \frac{a-\alpha+\e}{a+\alpha+\e }\beta (t)+2\alpha \int_t^\infty \beta \ ,$$
we finally obtain
$$\forall t\ge 0\ ,\ \beta (t)\leq K'\, {\rm e}^{-(a+\alpha -\e )t}$$
with $K'=K'(\alpha, M, C,\e )$.
We have the same estimate for $|X(t)|$, and, coming back to the identity
$$X_\infty =X(0)+\int_0^\infty {\rm e}^{sA}[Q(X(s)]\, ds\ ,$$
we conclude, by choosing $ \e $ small enough, that $$|X_\infty |\leq C'(\alpha, M, C).$$
\end{proof}
Notice that Lemma \ref{bdz} combined with the estimate
$$|c(t)-c_\infty {\rm e}^{-itM}|=O\left (\int_t^\infty\beta (s)\, ds \right )$$
leads to
$${\rm dist}(u(t), \mathcal C_M)=O\left ({\rm e}^{-\lambda t }\right )\ ,\ \lambda =\frac{a+\alpha}{2}\ .$$
Finally, in order to describe the geometric structure of $\Sigma_{M,\alpha}$, we give a complete description of the asymptotic properties of $u(t)$ as $t\to +\infty$.
\begin{lemma}\label{bcp}
Under the assumptions of Lemma \ref{bdz}, there exist $(\beta_\infty, \theta, \varphi )\in (0,\infty)\times \T\times \T $ such that, as $t\to +\infty$,
\begin{eqnarray*}
b(t)&\sim &  \sqrt{\beta_\infty}\, {\rm e}^{-\frac{a+\alpha}{2}t-itM\left (1+\frac{\alpha }{a}\right )+i\varphi}\ ,  \\
c(t)&\sim & \sqrt M {\rm e}^{-itM+i\theta}\ ,\\
p(t)&\sim & \sqrt{\frac{\beta_\infty}{M}}\left (\frac{\sqrt{a^2-\alpha ^2}-i(\alpha -a)}{2M}-1\right )\, {\rm e}^{-\frac{a+\alpha}{2}t+itM\frac{\alpha }{a}+i(\theta -\varphi)}\ .
\end{eqnarray*}
Conversely, for every $(\beta_\infty, \theta, \varphi )\in (0,\infty)\times \T\times \T $, there exists a unique trajectory 
$$u(t,x)=b(t)+\frac{c(t){\rm e}^{ix}}{1-p(t){\rm e}^{ix}}$$
satisfying the above asymptotic properties.
\end{lemma}
\begin{proof}
The asymptotic property has already been established for $c(t)$. Let us prove it for $b(t)$ and $p(t)$.
Recall  from system \eqref{system} that
\begin{eqnarray*}
i(\dot b+\alpha b)&=&(|b|^2+2M(1-|p|^2))b +Mc\overline p=\left (\beta +2M-2\delta +\frac{\zeta}{\beta}\right )b\ ,\\
i\dot p&=&M(1-|p|^2)p+c\overline b=\left (M-\delta +\frac{\zeta}{\delta}\right )p\ .
\end{eqnarray*}
Furthermore, we know that $\beta, \delta, \zeta $ satisfy the properties of Lemma \ref{bdz}. Therefore
\begin{eqnarray*}
i\frac{d}{dt}\left (\frac{b}{\sqrt\beta}\right )&=&\left (\beta +2M-2\delta +\frac{{\rm Re}\zeta}{\beta}\right )\frac{b}{\sqrt \beta}\\
&=&\left (M\frac{a+\alpha}{a}+O\left ({\rm e}^{-(a+\alpha)t}\right )\right )\frac{b}{\sqrt \beta}\ ,\\
i\frac{d}{dt}\left (\frac{\sqrt M p}{\sqrt \delta}\right )&=&\left ( M-\delta +\frac{{\rm Re}\zeta}{\delta}\right )\frac{\sqrt M p}{\sqrt \delta}\\
&=&\left (-\frac{M\alpha }{a} +O\left ({\rm e}^{-(a+\alpha)t}\right )\right )\frac{\sqrt M p}{\sqrt \delta}\ .
\end{eqnarray*}
This implies that there exist angles $\varphi, \psi$ such that
\begin{eqnarray*}
b(t)&\sim &  \sqrt{\beta_\infty}\, {\rm e}^{-\frac{a+\alpha}{2}t-itM\left (1+\frac{\alpha }{a}\right )+i\varphi}\\
p(t)&\sim & \sqrt{\frac{\beta_\infty}{M}}\left (\frac{a-\alpha}{a+\alpha}\right )^{\frac 12}\, {\rm e}^{-\frac{a+\alpha}{2}t+itM\frac{\alpha }{a}+i\psi }
\end{eqnarray*}
In view of Lemma \ref{necessary}, 
$${\rm e}^{-itM}\sqrt M\, \frac{\overline {p(t)}}{b(t)}\to {\rm e}^{-i\theta }\,
\left (\frac{\sqrt{a^2-\alpha ^2}+i(\alpha -a)}{2M}-1\right )$$
and of the elementary formula
$$\left (\frac{a-\alpha}{a+\alpha}\right )^{\frac 12}=\left | \frac{\sqrt{a^2-\alpha ^2}-i(\alpha -a)}{2M}-1    \right |\ ,$$
we infer that the asymptotic formula for $p(t)$ reads in fact
$$p(t)\sim  \sqrt{\frac{\beta_\infty}{M}}\left (\frac{\sqrt{a^2-\alpha ^2}-i(\alpha -a)}{2M}-1\right )\, {\rm e}^{-\frac{a+\alpha}{2}t+itM\frac{\alpha }{a}+i(\theta -\varphi)}\ .$$
Conversely, let us prove that, given any $(\beta_\infty ,\theta ,\varphi)\in (0,\infty)\times \T\times \T $, there exists a unique trajectory with these asymptotic properties. By Lemma \ref{bdz}, there exists a unique trajectory $(\beta ,\delta ,\zeta )$ of the reduced system such that
\begin{eqnarray*}
\beta (t)&=&\beta_\infty \, {\rm e}^{-(a+\alpha)t}\left (1+O\left ({\rm e}^{-(a+\alpha )t }\right )\right )\ ,\\
\delta (t)&=&\frac{a-\alpha}{a+\alpha}\beta_\infty \, {\rm e}^{-(a+\alpha)t}\left (1+O\left ({\rm e}^{-(a+\alpha )t }\right )\right )\ ,\\
\zeta (t)&=&\left (\frac{\sqrt{a^2-\alpha^2}+i(\alpha -a)}{2}-M\right )\beta_\infty \, {\rm e}^{-(a+\alpha)t}\left (1+O\left ({\rm e}^{-(a+\alpha )t }\right )\right )\ .
\end{eqnarray*}
Note that ${\rm e}^{2\alpha t}[|\zeta |^2-(M-\delta )^2\beta \delta ]$ is a constant wich tends to $0$ as $t\to +\infty$, hence it is identically $0$. This implies $\beta >0, 0\leq \delta $. Fix $T>0$ big enough so that 
$M>\delta (T)>0, \zeta (T)\ne 0$.   Consider the solution $(b_1,c_1,p_1)$ of the system in $(b,c,p)$ with
$$b_1(T)=\sqrt{\beta(T)}\ ,\  \sqrt M\, p_1(T)=\sqrt{\delta (T)}\ ,\  M\, c_1(T)=\frac{\zeta (T)}{\overline{b_1(T)p_1(T)}}\ .$$
Then, by uniqueness of the Cauchy problem for the reduced system, we have
$$\forall t\in \R\ ,\ |b_1(t)|^2=\beta (t)\ ,\ M\, |p_1(t)|^2=\delta (t)\ ,\ M\, c_1(t)\overline {b_1(t)p_1(t)}=\zeta (t)\ .$$
Applying the first part of Lemma \ref{bcp}, there exists $(\theta_1,\varphi_1)\in \T \times \T$ such that
as $t\to +\infty$,
\begin{eqnarray*}
b_1(t)&\sim &  \sqrt{\beta_\infty}\, {\rm e}^{-\frac{a+\alpha}{2}t-itM\left (1+\frac{\alpha }{a}\right )+i\varphi_1}\ ,  \\
c_1(t)&\sim & \sqrt M {\rm e}^{-itM+i\theta_1}\ ,\\
p_1(t)&\sim & \sqrt{\frac{\beta_\infty}{M}}\left (\frac{\sqrt{a^2-\alpha ^2}-i(\alpha -a)}{2M}-1\right )\, {\rm e}^{-\frac{a+\alpha}{2}t+itM\frac{\alpha }{a}+i(\theta _1-\varphi_1)}\ .
\end{eqnarray*}
Then 
$$b(t):={\rm e}^{i(\varphi -\varphi_1)}b_1(t)\ ,\ c(t):={\rm e}^{i(\theta -\theta_1)}c_1(t)\ ,\ p(t):={\rm e}^{i(\theta -\varphi -\theta_1+\varphi_1)}p_1(t)$$
satisfies the system in $b,c,p$ with the required asymptotic properties.

Finally, let us prove the uniqueness of such a solution. If $\tilde b, \tilde c, \tilde \zeta $ is a solution of the same system with the same asymptotic properties, we first observe that, in view of the uniqueness in Lemma \ref{bdz}, 
$$  \forall t\in \R\ ,\ |\tilde b(t)|^2=\beta (t)\ ,\ M\, |\tilde p(t)|^2=\delta (t)\ ,\ M\, \tilde c(t)\overline {\tilde b(t)\tilde p(t)}=\zeta (t)    \ .$$
Then we come back to the equations on $b/\sqrt{\beta }$ and $\sqrt Mp/\sqrt{\delta }$ that we derived in the beginning of this proof. We obtain
\begin{eqnarray*}
i\frac{d}{dt}\left (\frac{b-\tilde b}{\sqrt\beta}\right )
&=&\left (M\frac{a+\alpha}{a}+O\left ({\rm e}^{-(a+\alpha)t}\right )\right )\frac{b-\tilde b}{\sqrt \beta}\ ,\\
i\frac{d}{dt}\left (\frac{\sqrt M (p-\tilde p)}{\sqrt \delta}\right )&=&\left (-\frac{M\alpha }{a} +O\left ({\rm e}^{-(a+\alpha)t}\right )\right )\frac{\sqrt M (p-\tilde p)}{\sqrt \delta}\ .
\end{eqnarray*}
This implies that  $\xi _1:=(b-\tilde b)/\sqrt{\beta }$ and $\xi _2:=\sqrt M(p-\tilde p)/\sqrt{\delta }$ satisfy the inequality
$$|\xi (t)|\leq \int_t^\infty O\left ({\rm e}^{-(a+\alpha)s}\right )|\xi (s)|\, ds$$
and consequently that $\xi (t)\equiv 0$ for $t$ large enough. Hence $b=\tilde b, p=\tilde p$, and finally
$c=\tilde c$ from the definition of $\zeta $.\end{proof}
In order to complete the proof of Theorem \ref{Wbis}, we  consider the mapping
$$J:(0,\infty)\times \T \times \T \longrightarrow \mathcal E_M$$
defined by 
$$J(\beta_\infty ,\theta ,\varphi )=u(0)\ ,$$
where $u$ is the unique solution of \eqref{DS} provided by the second statement of Lemma \ref{bcp}. 
In view of Lemma \ref{bcp}, the range of $J$ is precisely the set $\Sigma_{M,\alpha}$. In order to prove that this set is a submanifold of dimension $3$ of $\mathcal E_M$, it is enough to establish that $J$ is a one to one proper immersion. 

The injectivity of $J$ is trivial. Its smoothness with respect to $\beta_\infty$ is a consequence of the fixed point argument in Lemma \ref{bdz} ; the dependence with respect to $(\theta,\varphi )$ is much more elementary, since it reflects the gauge and translation invariances, hence it is smooth as well. 

To prove the immersion property, we just have to check that, for every $(\beta_\infty ,\theta ,\varphi )$, the  three vectors $$\partial_{\beta_\infty}J(\beta_\infty ,\theta ,\varphi ), \partial_\theta J(\beta_\infty ,\theta ,\varphi ), \partial_\varphi J(\beta_\infty ,\theta ,\varphi )$$
are independent. We claim that the subspace spanned by these three vectors is also spanned by $\partial_tu(0), -i\partial_xu(0), iu(0)$. Indeed, in view of Lemma \ref{bcp} and of the invariances of equation \eqref{DS}, one easily checks the following identity,
\begin{eqnarray*}
&&{\rm e}^{i\varphi}S_\alpha(t+T)\left [ J(\beta_\infty, \theta_0, \varphi_0)\right ](x+\theta -\varphi))=\\
&&S_\alpha(t)\left [ J\left (\beta_\infty{\rm e}^{-(a+\alpha)T}, \theta_0+\theta -MT, \varphi_0+\varphi -MT\left (1+\frac{\alpha}{a}\right )\right )\right ](x)\ .
\end{eqnarray*}
 If these three vectors were dependent, this would mean that $u$ is a traveling wave of equation \eqref{DS}. This would impose that $(u\vert 1)\equiv 0$, hence $u\in \mathcal C_M$, which is impossible since $\Sigma_{M,\alpha}$ is disjoint from $ \mathcal C_M$. 
 
 Finally, $J$ is proper because of the last statement of Lemma \ref{bdz}.
\s
The last statement to be proved is that $\Sigma_{M,\alpha}\cup \mathcal C_M$ is closed. 
Since $\mathcal C_M$ is compact, it is enough to prove that the closure of
$\Sigma_{M,\alpha}$ is contained into  $\Sigma_{M,\alpha}\cup \mathcal C_M$. Let $(u_n)$ be a sequence of points of $\Sigma_{M,\alpha}$ which tends to $u\in \mathcal W$.
Set $(\beta_n,\theta_n,\varphi_n):=J^{-1}(u_n)$. Since $J$ is a homeomorphism onto its range $\Sigma_{M,\alpha}$, the only cases to be studied are $\beta_n\to 0$ and $\beta_n\to +\infty $. Now we appeal to the last statement of Lemma \ref{bdz}. In the first case, we obtain that $(u\vert 1)=0$, and more generally that $(S_\alpha (t)(u)\vert 1)=0$, so that $u\in \mathcal C_M$. In the second case, we infer $|(u_n\vert 1)|\to \infty $, which 
contradicts the fact that $u_n$ is convergent.
\s
The proof of Theorem \ref{Wbis} is  complete.

\appendix 
\section{Numerical simulations\\ in collaboration with C. Klein\\ (Universit\'e de Bourgogne)}
As mentioned in the introduction, the complete study made in section 4 suggests 
that the open subset $\Omega $ of initial data which give rise to 
exploding orbits is a dense subset. Furthermore, it is natural to ask 
about the rate of the Sobolev norms for such trajectories. For 
instance, is it true that the square of the $H^1$ norm grows linearly 
for generic initial data, as in the case of exploding trajectories in 
$\mathcal W$? The numerical simulations below suggest that these two 
questions have a positive answer.    

To numerically study the damped Szeg\"o equation (\ref{DS}), we 
approximate $u$ by a trigonometric polynomial, 
\begin{equation}
    u(x)\approx  \sum_{k=-N/2+1}^{N/2}\hat{u}_{k}e^{ikx},\quad N\in 
    2\mathbb{N}
    \label{ufft},
\end{equation}
i.e., we consider the \emph{discrete Fourier transform} (DFT) of a vector $u$ (in 
an abuse of notation, we use the same symbol for the function $u$ and 
its discrete approximation) with components $u_{n}=u(x_{n})$, where 
$x_{n}=-\pi+2\pi n/N$, $n=1,2,\ldots,N$. The DFT can be 
computed efficiently with a \emph{Fast Fourier transform}. Note that we work in the 
Hardy space, thus all coefficients $\hat{u}_{k}$ corresponding to 
negative wave numbers vanish. But for computing the DFT, the negative 
wave numbers are nonetheless important. The action of the projector $\Pi$ is 
simply to put the coefficients of the negative wave numbers  
equal to zero.

The approximation of the function $u$ via a DFT 
implies that equation (\ref{DS}) is approximated via a finite 
dimensional system of ODEs. The latter is integrated with the standard 
explicit fourth order Runge-Kutta method.  Derivatives with respect to $x$ are computed in 
standard way by multiplying the coefficients $\hat{u}_{k}$ by $ik$. 
We apply a \emph{Krasny filter}, i.e., we put Fourier 
coefficients with a modulus smaller than $10^{-12}$ equal to 0 to 
address that we work with finite precision, and to take care of 
unavoidable rounding errors.

To test the code, we first study an  element of $\mathcal W$, namely
$$u_0(x)=\frac{{\rm e}^{ix}}{1-p{\rm e}^{ix}}$$
with $p=0.5$ and $\alpha =1$, for which we validate the result of Theorem \ref{Wbis}, namely, as $t\to +\infty $, 
\begin{equation}\label{H1}
\Vert S_\alpha (t)u_0\Vert _{H^1}^2\sim \frac{4\alpha M^3}{\alpha ^2+M^2} t\ .
\end{equation}
Notice that $$\Vert u_0\Vert_{L^2}^2=\frac{1}{1-|p|^2}<\frac{1}{(1-|p|^2)^2}=M(u_0)\ .$$
We use $N_{t}=10^{5}$ time steps for $t\in[0,20]$ and $N = 2^{12}$. 
During the computation we observe relative conservation of the 
momentum to the order of $10^{-15}$, i.e., essentially machine 
precision, and the Fourier coefficients decrease to the order of the 
Krasny filter. This means the solution is well resolved both in space 
and in time. As can be seen in Fig.~\ref{figszegoetest}, the 
agreement between theoretical prediction and numerics is excellent, 
the asymptotic regime is reached for comparatively small values of 
$t$. 
\begin{figure}[htb!]
  \includegraphics[width=0.7\textwidth]{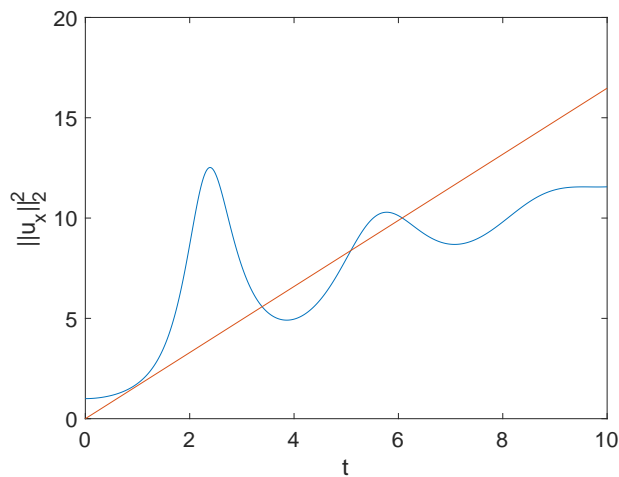}
 \caption{In blue, the norm $||S_\alpha(t)u_0||_{H_1}^2$ for the solution to 
 the equation (\ref{DS}) with $\alpha=1$ for the initial data 
 $u_0(x)=\frac{{\rm e}^{ix}}{1-p{\rm e}^{ix}}$ with $p=0.5$ in 
 dependence of time.   In red, the asymptotic relation (A.2) }
 \label{figszegoetest}
\end{figure}

The second example studies the case of an initial datum with two poles,
$$u_0(x)=\frac{{\rm e}^{ix}}{1-p_1{\rm e}^{ix}}+\frac{{\rm e}^{ix}}{1-p_2{\rm e}^{ix}},$$
which corresponds to a more complicated phase space than $\mathcal 
W$, but still finite dimensional, since the rank of $K_{u_0}$ is $2$. 
We put $p_{1}=0.7$, $p_{2}=0.8$ and use the same numerical parameters 
as before. The relative conservation of the momentum is of the order 
of $10^{-7}$. The norm $||S_\alpha(t)u_0||_{H_1}^2$ in dependence of 
time for this example can be seen in Fig.~\ref{figszegoetest2}. The 
norm appears to grow linearly in time. 
\begin{figure}[htb!]
  \includegraphics[width=0.7\textwidth]{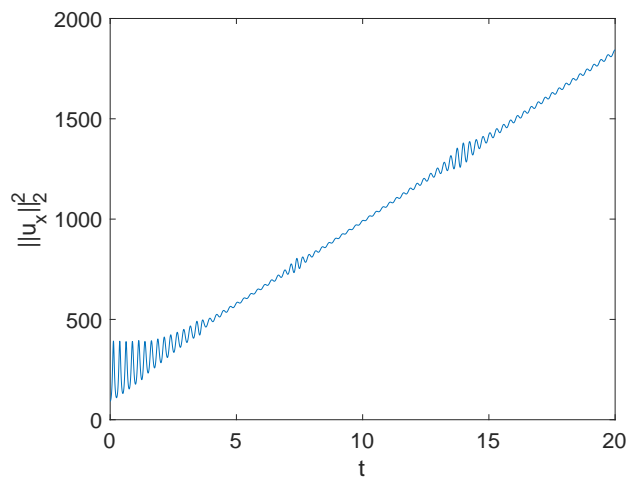}
 \caption{The norm $||S_\alpha(t)u_0||_{H_1}^2$ for the solution to 
 the equation (\ref{DS}) with $\alpha=1$ for the initial data 
 $u_0(x)=\frac{{\rm e}^{ix}}{1-0.7{\rm e}^{ix}}+\frac{{\rm 
 e}^{ix}}{1-0.8{\rm e}^{ix}}$ in 
 dependence of time.}
 \label{figszegoetest2}
\end{figure}

The third example illustrates the case of an arbitrary initial datum 
with a Gaussian profile, $u_{0}=\Pi\exp(-10x^{2})$. We use the same 
numerical parameters, this time for $t\in[0,1000]$. The momentum is 
conserved to the order of $10^{-13}$.  The norm 
$||S_\alpha(t)u_0||_{H_1}^2$ can be seen for this case in 
Fig.~\ref{figszegoetest3}. Once more the square of the $H^1$ norm 
grows linearly in time.
\begin{figure}[htb!]
  \includegraphics[width=0.7\textwidth]{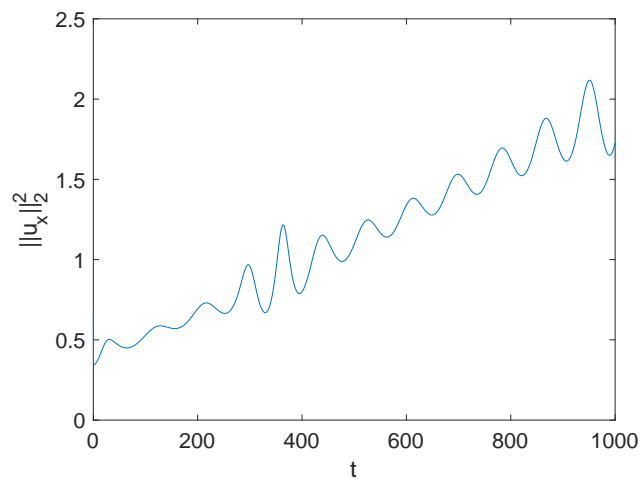}
 \caption{The norm $||S_\alpha(t)u_0||_{H_1}^2$ for the solution to 
 the equation (\ref{DS}) with $\alpha=1$ for the initial data 
 $u_0(x)=\Pi\exp(-10x^{2})$ in 
 dependence of time.}
 \label{figszegoetest3}
\end{figure}

\eject

\end{document}